\documentclass{article}
\usepackage{graphicx}
\usepackage{amsmath}
\usepackage{amsfonts}
\usepackage{amssymb}
\usepackage{amsthm}
\usepackage{mathtools}
\usepackage{enumerate}
\usepackage{lscape}
\usepackage{longtable}
\usepackage{rotating}
\usepackage{multirow}
\usepackage{url}
\usepackage{subfigure}
\usepackage{rotating}
\usepackage{comment}
\usepackage[utf8]{inputenc}
\usepackage{amssymb}
\usepackage{amsfonts}
\usepackage{tikz}
\usepackage{natbib} 

\usepackage{color}
\usepackage{hyperref}

\usepackage{algorithm}
\usepackage{algpseudocode}
\def\fct#1{\mathop{\rm #1}}	

\def\Argmin {\fct{argmin}}
\def\spc{~~~}
\newtheorem{theorem}{Theorem}[section]

\theoremstyle{definition}

\newcommand{\R}{\mathbb{R}}

\title{Towards Real Time Control of Water Engineering with Nonlinear Hyperbolic Partial Differential Equations}
\author{Fabio Difonzo, Michael Holst, Morteza Kimiaei, \\ Vyacheslav Kungurtsev,  Songqiang Qiu}
\date{December 2025}

\begin{document}
\maketitle

\begin{abstract}
This paper examines aspirational requirements for software addressing mixed-integer optimization problems constrained by the nonlinear Shallow Water partial differential equations (PDEs), motivated by applications such as river-flow management in hydropower cascades. Realistic deployment of such software would require the simultaneous treatment of nonlinear and potentially non-smooth PDE dynamics, limited theoretical guarantees on the existence and regularity of control-to-state mappings under varying boundary conditions, and computational performance compatible with operational decision-making. In addition, practical settings motivate consideration of uncertainty arising from forecasts of demand, inflows, and environmental conditions.

At present, the theoretical foundations, numerical optimization methods, and large-scale scientific computing tools required to address these challenges in a unified and tractable manner remain the subject of ongoing research across the associated research communities. Rather than proposing a complete solution, this work uses the problem as a case study to identify and organize the mathematical, algorithmic, and computational components that would be necessary for its realization. The resulting framework highlights open challenges and intermediate research directions, and may inform both more circumscribed related problems and the design of future large-scale collaborative efforts aimed at addressing such objectives.
\end{abstract}

\tableofcontents
\begin{sloppypar}

\section{Introduction}

Consider that one wishes to control a system defined by a nonlinear time-dependent Partial Differential Equation (PDE). Moreover, the control decisions include both continuous and discrete variables; the presence of discrete variables induces a nonconvex feasible set and hence additional nonconvexity in the optimization problem. In addition, the control is intended to be performed in real-time, that is, with control decisions algorithmically computed within a short time frame. Finally, some quantities that influence the dynamics or the desired output of the system are uncertain. These uncertainties have probabilistic forecasts, and decisions must be made with this uncertainty in mind, whether by targeting an expectation of the desired outcome and/or a quantity representing risk aversion. 

Such a problem arises in a number of applications. A prominent example is in hydrology, as in the operation of a dam or network of dams in service of generating hydropower to meet energy demand, which can subside or surge in real time. So far, nonlinear control algorithms have been developed applying a nonlinear hydropower production function (e.g.,~\cite{hamann2015coordinated}) as well as empirical studies on, e.g., efficiency~\cite{zheng2013evaluation}. However, being able to successfully incorporate accurate models of the river flow, together with important physical (e.g., environmental) constraints, can facilitate a more realistic model and thus more simultaneously efficient and reliable operation. 

One can also consider groundwater for the irrigation of a set of plots of land, with some growing agricultural produce and others being pasture for a range of livestock. By considering the porosity of the soil and modeling the water source input from precipitation and river flow at different times of the year, one can aim to direct it towards robust storage for possible drought periods, as well as the ideal irrigation for maximizing agricultural yield (see, e.g., \cite{Berardi_Difonzo_Guglielmi_2023,AllaBerardiSaluzzi2025}).

These two examples will be the focus of this paper, with particular attention devoted to the first. However, there are many other circumstances and domains in which this problem can occur. For instance, in automated manufacturing and process control, a number of operations can be modeled with evolution PDEs. For instance, a chemical refinery plant could perform processing that can be modeled thermodynamically with a heat equation plus an additional nonlinear dissipative term.

In general, this is an intractable problem. That is, the combination of a nonlinear infinite-dimensional system, nonconvexities in the decisions, uncertainty, and fast computation for real-time solution is a set of requirements that far exceed the capability of state-of-the-art methods and software. Whereas an aspirational volume on real-time PDE-constrained optimization methods has appeared even in the early volume~\cite{biegler2007real}, in general, there are very few works in the literature with real implementation or even realistic simulations under the defined operational requirements. For a rare exception, see, e.g.,~\cite{behrens2014real}, wherein while the system is multi-component, Newton iterations are numerically reliable, facilitating real-time computation, although still uncertainty and combinatorial decisions are not considered in that work. Other recent attempts at real-time control of systems governed by PDEs incorporate neural network models (e.g.,~\cite{wang2021fast}); however, these often fail to generalize, and their fundamental black box nature and stochasticity present unacceptable hazards for any mission-critical system. 

However, while both methods and software are a significant distance from being available to adequately solve nonlinear time-dependent PDE mixed-integer optimization, this paper presents a framework and overall methodology whereby this problem can be approached over the long run by researchers in the field. In particular, through a properly motivated and reasonable incorporation of existing tools, incremental extensions of available methods and analysis, together with a few larger research and development efforts to circumnavigate certain fundamental obstacles, one can attempt to tackle such a control problem. 

This paper presents the following contributions in the direction of solving this problem. While it is not authoritatively comprehensive due to the significant research and computational developments necessary to realize the entire proposed operation, this manuscript can form the basis of a large, potentially multi-institutional formal project or serve as a template by which decentralized advances can yield, in the long run, a workable solution. 
\begin{itemize}
    \item We present a theoretical analysis of the PDE under consideration, in particular the appropriate function space to consider solutions. A novel local existence theory is presented for the nonlinear nonconservative form of the shallow water equations. The required challenges, as far as the lack of a global unique solution and a dual space amenable to computation, are defined, motivating the procedures in the rest of the paper.
    \item A novel three-tier computational architecture is developed. This includes an offline layer for exploiting HPC and machine learning, a meso layer for performing classical Approximate Dynamic Programming and Nonlinear-Optimization-based Model Predictive Control, and a real-time layer for online operation, together with desiderata and protocols for their integration.
    \item The details of offline layer are presented, including a novel use of machine learning tools to find a region of the control-to-state map on which local existence and regularity permits numerical optimization.
    \item The details of a meso layer that combines particle evolution strategies together with a novel Primal (due to the aforementioned dual space challenges) Sequential Quadratic Programming with Quadratic Constraints Algorithm for performing Nonlinear Model Predictive Control
    \item The appropriate tools for tractable and reliable real-time operation are proposed, with reduced order modeling used to define tractable optimization problems for online operation, with feedback at reasonable latency across the layers to enable maximal instrumental utility of the more complex computations at the meso and offline layers.
\end{itemize}

The computational systems architecture presented is not only multilevel to account for limitations in real-time hardware latency, but the motivating application of dam networks naturally involves multiple, well-separated operational time scales. Long-term planning decisions, such as reservoir allocation and policy design, are typically performed over horizons of days to weeks. Intermediate supervisory adjustments, incorporating updated inflow forecasts and system measurements, occur on the scale of hours, for example through rolling-horizon updates every 1–6 hours. At the fastest level, hydraulic control actions such as gate movements and flow regulation must respond to system dynamics on time scales ranging from seconds to minutes. This separation of time scales directly motivates the three-layer computational architecture developed in this work, consisting of offline, mesoscopic, and real-time components.

\section{Formal Problem Statement and Assumptions}

To begin with, the problem is stated without considering noise and uncertainty, i.e., deterministic PDE-constrained optimization. 

Formally,
\begin{equation}\label{eq:pdeopt}
\begin{array}{rl}
\min\limits_{u,v,z} & \int_{t=0}^T \left(-\Pi(u,d)+C(v,z)\right) dt\\
\text{s.t. } & e(u(x,t)) = f(v(t),z(t)),\\ 
& v\in\mathcal{V},\\
& S(u)\ge 0.
\end{array}
\end{equation}
This defines a PDE-constrained optimization problem for a process $u$ unrolling over time $t$ and space $x$. In the standing example of interest, the process has multiple vector fields, the water level and the flow rate, however for simplicity we present the coupled PDE with $e(u(x,t)) = f(v(t),z(t))$, where $v(t)\in \mathbb{R}^{n_c}$ is a continuous finite dimensional control decision variable (e.g., the aperture of a turbine) and $z(t)\in\{0,1\}^{n_b}$, e.g., the opening and closing of a dam. 
In practice, the influence of the decision variables on the state arises from the boundary conditions, the details of which are presented below. 

The functional $\Pi$ is an Economic Objective, typically the profit. Economic and more classical tracking objective functionals present important distinctions for Model Predictive Control, with stability guarantees requiring distinct problem formulations and analysis. A tracking objective, i.e., $\int \|u-u_T\|^2  dt$, for an off-line computed setpoint $u_T$ is strongly convex, presenting a natural Lyapunov function facilitating stability, while Economic objective functionals do not. 

The cost function $C$ for the decisions $v$ and $z$ represents real costs of operating the water engineering, and the function also acts as a regularizer. Here, there is a fortunate correspondence between economic cost and regularization: both the total time and effort in operation, that is, the time-measure of an open dam and its wide opening, as well as its time rate of change, that is, the process of opening and closing, are penalized. The sum of the total norm of the two together is effectively Sobolev-like. 

The constraint $\mathcal{V}$ is typically simple, e.g., lower and upper bounds on the components, for the continuous decisions. Fortunately, it is both simultaneously natural and physical, while also advantageous numerically, for the continuous decisions to be constrained to a compact set, e.g., $0\le v\le \bar{v}$ for (almost) every $x$. Operationally, there are clear limits present for any driving force, while compactness facilitates the lower semicontinuity properties that are amenable to iterative procedures asymptotically converging to a minimum.

Finally, $S(u)$ is a \emph{state constraint} that can require that the solution of the water flow exhibits certain physical requirements. This can include a) physical realizability, e.g., a nonnegative water height level, b) safety, e.g., non-flooding velocity, and c) externally mandated operational requirements, e.g., environmental constraints ensuring sufficient flow for fish migration.
The state constraints present additional challenges as far as regularity; the dual for such a constraint is a Borel measure. In the sequel, the suggested strategy for the state constraint is to add a barrier term in the objective penalizing proximity of $S(u)$ to its bound. In practice, this would ensure a more conservative constraint adherence (that is, at a greater distance from the bound) than is necessary, which is reasonable for mission-critical systems.

\subsection{River Flow Modeling}

The shallow-water (Saint--Venant) family of systems considered in this work may be regarded as a hierarchy of reduced fluid models whose fidelity and numerical requirements vary with the physical processes retained \cite{LeVeque_2002,toro2009}. For the purposes of control and real-time optimization, it is useful to highlight the modeling choices that most directly affect numerical solvability, robustness, and cost. At the lowest complexity, the one-dimensional depth-averaged Saint--Venant (conservative) formulation provides a hyperbolic balance law for cross-section area (or depth) and discharge. Two-dimensional depth-integrated shallow-water equations extend this model to account for lateral velocity components and nontrivial bathymetry. When dispersion or non-hydrostatic pressure becomes important (for instance, during runup or for short waves interacting with bathymetry), non-hydrostatic or multilayer extensions should be considered and treated with corresponding numerical care. See \cite{Bear_Verruijt} for more details.

Practical river and dam models include topographic source terms (bed slope), frictional source terms (e.g., Manning or Darcy--Weisbach), and other nonconservative products (for example, parameterized exchange or porosity terms). Such source terms can be stiff and are often spatially nonsmooth (for example, step changes in bed elevation or wetting/drying fronts). These characteristics require numerical methods that carefully balance flux gradients against source terms and that preserve physically meaningful invariants such as mass and nonnegativity of the depth.

Robust numerical solvers for hydrological and river Computational Fluid Dynamics (hereafter, CFD) that are suitable for control integration generally satisfy three interrelated properties. First, they are well-balanced: they preserve steady states in which flux gradients are exactly balanced by source terms (for example, the lake-at-rest state), which prevents spurious waves and allows accurate long-time integration around controller operating points \cite{audusse2004}. Second, they are positivity-preserving and handle wetting/drying in a stable manner so that water depth remains nonnegative even during transitions between inundated and dry states \cite{kurganov2007}. Third, they treat friction and other stiff sources in a stable way; common remedies include semi-implicit source discretizations or operator splitting with local implicit solves when tight stability constraints appear.

Finite volume (conservative) discretizations on structured or unstructured meshes are the dominant choice for hyperbolic shallow-water problems, since they provide local conservation, shock-capturing, and natural coupling with Riemann solver-based fluxes (e.g., see \cite{LeVeque_2002,toro2009}). 
High-order finite-volume or discontinuous Galerkin methods are used where higher accuracy is required in smooth regions, but they demand stronger limiters and increase computational cost. Mixed finite-element and stabilized continuous-Galerkin formulations are also employed in some operational hydrodynamic models because of their flexibility on complex geometries.

Hydrostatic reconstruction techniques and positivity-preserving limiters provide a robust way to treat bed-slope source terms and wetting/drying indicators of the form $H(\eta,z_b)=\max\{0,z_b+\eta\}$ (e.g., see \cite{BrugnanoCasulli2008,BrugnanoCasulli2009} and Section \ref{sec:pwLinearApproxSWE} below). Particular care is required at thin water layers near dry states: specialized reconstructions or local mesh refinement help avoid oscillations and preserve mass and nonnegativity \cite{kurganov2007}.

When primitive (nonconservative) formulations are used, the numerical flux must be combined with a consistent discretization of nonconservative products; path-conservative schemes provide a rigorous framework for this purpose \cite{pares2006,castro2019}. A careful treatment of these terms is important for accuracy and for ensuring that discrete adjoint and linearized models (used by control and optimization routines) are consistent with the forward solver.

For gradient-based optimization and real-time Newton/SQP solvers, it is necessary to obtain consistent linearizations of the discrete forward solver or to use algorithmic differentiation of the numerical method to produce accurate sensitivities.

\paragraph{Saint Venant System}
Consider the Saint Venant system, which depth averages the Shallow Water Equations to produce a one-dimensional nonlinear hyperbolic PDE, as follows
\begin{equation}\label{eq:saintv}
    \begin{array}{l}
    \frac{\partial A}{\partial t}+\frac{\partial Q}{\partial x} = 0,\\
    \frac{\partial Q}{\partial t} +\frac{\partial(Q^2/A+g A^2/2)}{\partial x}=gA(S_o-S_f),\text{ with,} \\
    S_o := -\frac{\partial z_b}{\partial x} \\
    S_f := \frac{n_M^2 Q |Q|}{A^{10/3}} ,\\
    A(x,0)=A_0(x),\,Q(x,0)=Q_0(x).
    \end{array}
\end{equation}
Above, the function $A(x,t)$ denotes the cross-sectional area of the flow, which in one-dimensional shallow-water modeling is typically interpreted as a proxy for water depth under suitable geometric assumptions; the quantity $Q(x,t)$ represents the discharge, i.e., the volumetric flow rate; the constant $g$ denotes the gravitational acceleration; the bed elevation is given by $z_b(x)$; the friction term $S_f$ is modeled using Manning's empirical law, where $n_M$ denotes the Manning roughness coefficient; finally, $A_0(x)$ and $Q_0(x)$ specify the initial conditions for the cross-sectional area and discharge, respectively.

This system is an empirical approximation that is frequently used for modeling and simulations. It appears in~\cite{castro2019shallow}

Observe that the right-hand side has the nonlinear operator $-gAS_f=-\frac{gn^2_M Q\vert Q\vert}{A^{7/3}}$. This presents significant challenges in deriving any weak form existence result, as any Sobolev space for $A$ and $Q$ will not be closed under these nonlinear operations.

\paragraph{Two-Dimensional Shallow Water System}
We present an alternative formulation of the Shallow Water equations, extending the consideration to two spatial dimensions. This presentation exhibits more non-smooth maximum operations but fewer degenerate exponents. Consider:
\begin{equation}\label{eq:swe2d}
    \begin{array}{l}
    \frac{\partial u}{\partial t}+u \frac{\partial u}{\partial x}+g\frac{\partial \eta}{\partial x}=\frac{1}{H(\eta,z_b)}\left(\nu H(\eta,z_b) \frac{\partial u}{\partial x}\right)_x+\frac{\gamma_T u_A-\gamma u}{H(\eta,z_b)} \\
    \frac{\partial u}{\partial t}+u \frac{\partial u}{\partial y}+g\frac{\partial \eta}{\partial y}=\frac{1}{H(\eta,z_b)}\left(\nu H(\eta,z_b) \frac{\partial u}{\partial y}\right)_y+\frac{\gamma_T u_A-\gamma u}{H(\eta,z_b)} \\
    \frac{\partial \eta}{\partial t} + u \frac{\partial \eta}{\partial x}+H(\eta,z_b) \frac{\partial u}{\partial x} = - u \frac{dz_b}{dx} \\
    \frac{\partial \eta}{\partial t} + u \frac{\partial \eta}{\partial y}+H(\eta,z_b) \frac{\partial u}{\partial y} = - u \frac{dz_b}{dy} \\
    H(\eta,z_b) = \max\left\{0,z_b+\eta\right\},
    \end{array}
\end{equation}
where here $u$ is the velocity, $\eta$ represents the height of the water level. 




\subsection{Groundwater flow modeling and control}
In this case, the state constraint represents that under-saturation is undesirable, as it clearly leads to underwatering of agricultural production. Over-saturation is also undesirable as potentially damaging to flora; however, it is highly unlikely in an Arid region. 

Groundwater dynamics in porous media are governed by diffusion-dominated parabolic partial differential equations derived from Darcy’s law and mass conservation \cite{bear1972}. For a domain $\Omega \subset \mathbb{R}^d$ and hydraulic head $h(x,t)$, the classical form is
\[
S_s \frac{\partial h}{\partial t} - \nabla \cdot (K(x)\nabla h) = q(x,t),
\]
where $S_s$ is the specific storage coefficient, $K(x)$ is the hydraulic conductivity tensor, and $q(x,t)$ represents external sources or sinks such as pumping wells and recharge. Depending on the hydrogeological setting, one may consider confined flow (linear, single-valued head), unconfined flow with a free surface, or variably saturated flow described by the nonlinear Richards equation \cite{alt1983,celia1990,helmig1997}:
\[
\frac{\partial \theta(h)}{\partial t} - \nabla \cdot [K(h)\nabla(h+z)] = q(x,t),
\]
where $\theta(h)$ denotes water content and $K(h)$ the unsaturated conductivity. The Richards equation is strongly nonlinear and degenerate parabolic, posing analytical and numerical challenges when saturation fronts move or dry zones appear.

For the linear confined-flow case, existence, uniqueness, and stability follow from standard variational arguments, while the Richards equation requires monotone operator techniques to ensure existence of weak solutions \cite{alt1983}. From a control perspective, the state equation is well-posed for bounded source controls, but the control-to-state map may lose differentiability near saturation or free-surface transitions, complicating gradient-based optimization. Control variables typically represent pumping or injection rates, recharge operations, or boundary fluxes, and must respect physical and environmental constraints. 

Groundwater flow exhibits strong heterogeneity and parameter uncertainty, particularly in the hydraulic conductivity $K(x)$, which may vary over orders of magnitude. This uncertainty motivates stochastic formulations of the control problem \cite{carrera2005}, which are discussed in detail in Section \ref{sec:backgroundSPDEDP} and subsequent sections.

Discretizations typically rely on finite-element or finite-volume schemes that ensure local mass conservation and robustness for heterogeneous $K(x)$ fields. Mixed formulations are often preferred when flux accuracy is critical. For variably saturated flow, nonlinear iterations such as Picard or Newton methods with adaptive time stepping are required to handle the strong nonlinearity in $K(h)$ and $\theta(h)$ \cite{celia1990,delaunay2017}, while monotone schemes are necessary to ensure physically meaningful solutions.


For optimization and feedback design, adjoint equations for the linear case are classical parabolic PDEs, while nonlinear or unsaturated formulations require consistent linearizations of $K(h)$ and $\theta(h)$, achievable through algorithmic differentiation of the discrete solver. Reduced-order models (ROMs) based on Proper Orthogonal Decomposition (POD) or DEIM can accelerate repeated PDE evaluations \cite{willmann2010}.

\subsection{Relevant Background on PDE-Constrained Optimization} 

Here, we review some background on PDE-constrained optimization that is particularly applicable to the problem of interest. We focus on deterministic problems, as formally defined above. In the sequel, we present and review the appropriate literature for problems that incorporate uncertainty.
\subsubsection{On Control of Hyperbolic PDEs}\label{sec:controlHyperbolicPDE}

Optimal control of hyperbolic conservation laws differs fundamentally from the elliptic and parabolic cases, 
primarily because even entropy solutions of hyperbolic PDEs exhibit finite-speed propagation and naturally develop 
discontinuities (shocks, rarefactions) in finite time. Even in the case of smooth data and continuous controls (and considering relaxations or switching interpretations of the discrete controls $z$), the mapping $(v,z)\mapsto u$ is not classically differentiable: the dominant effect of a control perturbation is frequently a 
shift of shock locations rather than a smooth variation of the state. Hence, the standard control-to-state 
derivative fails, and sensitivity analysis must be performed in weak or measure-valued topologies. This distinction underlies many of the algorithmic and analytical challenges discussed throughout Section~2, and we therefore focus here on the features that are specific to hyperbolic and parabolic-hyperbolic dynamics.

A general controlled hyperbolic conservation law in conservative form reads
\[
\partial_t u(x,t) + \nabla_x\!\cdot F(u(x,t)) 
= S(u(x,t),v(t),z(t),x,t),
\qquad (x,t)\in\Omega\times[0,T],
\]
where $u$ is the conserved quantity, $F(u)$ is the nonlinear flux, and $S(u,v,z)$ encodes interior or boundary 
actuation through the continuous control $v$ and binary control $z$. The conservative formulation is essential, 
as it provides the correct weak (entropy) solution concept and the Rankine--Hugoniot relations governing shock propagation. 

We write 
\[
\Omega_T := \Omega \times (0,T)
\qquad\text{and}\qquad
L^\infty(\Omega_T)
\]
for the space of measurable functions on the space--time cylinder that are
essentially bounded, i.e.,
\[
L^\infty(\Omega_T)
=
\bigl\{
u:\Omega_T\to\mathbb{R}
\;\big|\;
\operatorname*{ess\,sup}_{(x,t)\in\Omega_T}|u(x,t)| < \infty
\bigr\}.
\]
Moreover, we write 
\[
C([0,T];L^1_{\mathrm{loc}}(\Omega))
\]
for the space of functions that are continuous in time with values in 
$L^1_{\mathrm{loc}}(\Omega)$; that is, for every compact $K\subset\Omega$,
\[
\int_K |u(x,t)-u(x,t_0)|\,dx \to 0 \quad\text{as }\; t\to t_0.
\]
$L^\infty(0,T)$ let denote the space of essentially bounded measurable functions
on $(0,T)$, while 
\[
L^\infty(0,T;\{0,1\})
= \{\,z:(0,T)\to\{0,1\}\ \text{measurable}\,\}
\]
denotes the space of measurable on/off (binary) switching controls.

\paragraph{Regularity of States, Controls, and Adjoints.} Ulbrich’s analysis~\cite{ulbrichpdeopthyper} shows that entropy solutions of controlled hyperbolic conservation laws satisfy
\[
u \in L^\infty(\Omega_T) \cap C([0,T];L^1_{\mathrm{loc}}(\Omega)),
\]
and are functions of bounded variation (BV) in space for almost every time. Away from shock curves, the solutions are piecewise smooth, but discontinuities propagate along characteristic directions.  
Even if $(v,z)$ are essentially bounded,
\[
v \in L^\infty(0,T), \qquad z\in L^\infty(0,T;\{0,1\}),
\]
the resulting state $u$ depends on $(v,z)$ in a highly nonsmooth way: small perturbations in the controls may lead to finite shifts in the shock locations. Hence, the mapping $(v,z)\mapsto u$
fails to be Gâteaux differentiable.

The adjoint equation corresponding to an optimal control problem with running cost $\ell(u,v,z)$
is a backward hyperbolic transport equation,
\[
-\partial_t p - [\nabla_u F(u)]^{\!\top}\nabla_x p
= \nabla_u S(u,v,z)^{\!\top}p + \nabla_u \ell(u,v,z),
\]
derived in~\cite{ulbrichpdeopthyper}. Ulbrich shows that the adjoint $p$ inherits the BV regularity of $u$: it is transported backward
along characteristics, remains piecewise smooth, and satisfies compatibility jump conditions across
the shocks of the forward solution. These jump conditions are essential for obtaining a consistent
sensitivity formula in the presence of discontinuities.

\paragraph{Shift-Differentiability and Sensitivity.} Because classical derivatives do not exist, Ulbrich introduced shift-differentiability for entropy solutions~\cite{ulbrichpdeopthyper}.  
A perturbation of the controls induces smooth variations in regions where $u$ is continuous, and shifts of shock curves weighted by the jump magnitude. This decomposition enables a rigorous directional derivative of the reduced cost functional even
when the solution contains shocks. Pfaff and Ulbrich extend this calculus to switched and hybrid
controls in~\cite{PfaffUlbrichTLP,Sulb2016}, where sensitivities also incorporate the influence of
switching times and discrete control actions.

\paragraph{Conservative vs.\ Nonconservative Form.}
Recall the distinction between a \emph{conservative} and \emph{nonconservative} force-driven evolution process. The latter incorporates forces that are dissipative, e.g., self-friction as well as friction with the bottom surface and air. Thus, a physically faithful approach to appropriately modeling the Shallow Water Equations as studied in hydrological engineering~\cite{castro2019shallow} requires consideration of the nonconservative hyperbolic form. 

Ulbrich’s framework~\cite{ulbrichpdeopthyper} applies strictly to conservation laws written in
conservative form,
\[
\partial_t u + \nabla\!\cdot F(u)=S,
\]
because entropy solutions, shock speeds, and the associated adjoint jumps are intrinsically defined
in this formulation.  
In contrast, the nonconservative form
\[
\partial_t u + A(u)\nabla u = S
\]
is not well defined across discontinuities unless an additional nonconservative product (in the
sense of Dal Maso--LeFloch--Murat) is specified.  For this reason, sensitivity and adjoint methods for control of hyperbolic PDEs have been formulated and
analyzed in the conservative setting. Here, the matrix $A(u)$ denotes the Jacobian of the flux,
\[
A(u) := \nabla_u F(u),
\]
so that in one spatial dimension the conservative law  $\partial_t u + \partial_x F(u) = S$ can be rewritten via the chain rule as the quasilinear form $\partial_t u + A(u)\,\partial_x u = S$.

\paragraph{Relevance for Hydrodynamic Control.} The Saint--Venant (shallow-water) equations used in this work constitute a system of hyperbolic
nonlinear PDEs with boundary and internal actuation.  Adjoint-based sensitivity methods developed in~\cite{Sulb2016,ulbrichpdeopthyper} and the switching-control extensions in~\cite{PfaffUlbrichTLP} can provide a potential tool, or base of a tool, to construct appropriate numerical methods. However, because of the nonconservative features of physical river flow, including self-friction and bottom friction, theoretical guarantees are not available. They can serve as one tool for the multi-scale stochastic control framework developed in the
subsequent sections; however, additional workarounds will be needed to consider the unknown topology of the control to state map and the lack of a dual space that is amenable to numerical approximation, in the case of the complete nonconservative modeling.

\subsubsection{On Mixed Integer PDE-Constrained Optimization}

We describe the structural aspects of the mixed-integer PDE-constrained control problem
arising from the hydropower cascade. The system consists of river segments
$u_{(r)}$, $r=1,\ldots,R$, and between two successive segments the dam $a$
is controlled by a set of continuous variables $\mathcal{N}_{a,v}\subset\{1,\dots,n_v\}$ 
and binary (on/off) switching variables $\mathcal{N}_{a,z}\subset\{1,\dots,n_z\}$, 
operating a boundary function 
\[
b_{(l)}(x)=B(v_{(j(l))},z_{(k(l))}), 
\qquad 
x\in\partial\Omega_{(l)},\quad l=1,\ldots,R-1,
\]
where $(j(l))\subset\mathcal{N}_{a,v}$ and $(k(l))\subset\mathcal{N}_{a,z}$.
The resulting optimization problem links discrete operating modes of dams with continuous
water-flow dynamics governed by the Saint--Venant or Shallow Water equations.

\paragraph{Stationary Mixed-Integer PDE-Constrained Optimization.}
Stationary mixed-integer PDE-constrained optimization extends classical finite-dimensional
MINLP techniques to steady-state PDE systems of the form
\[
e(u)=f(v,z),
\]
where $u$ denotes the PDE state, $v$ is a continuous control, and $z$ denotes binary
decision variables. Each discrete assignment $z$ requires solving a nonlinear stationary PDE.

{\bf Branch-and-bound} strategies for stationary PDE problems, such as those developed by
Hahn, Leyffer, and Zavala~\cite{hahn2017mixed}, embed relaxed continuous PDE-constrained subproblems at each node of the search tree. While the approach provides global guarantees on moderate-scale PDE models, evaluation of nonlinear PDE solves at every branch presents poor computational scaling.

{\bf Penalty and relaxation approaches} relax $z\in\{0,1\}$ to $z\in[0,1]$ and enforce
integrality using concave penalties,
\[
\frac{1}{\varepsilon}\sum_i z_i(1-z_i),
\]
as in Garmatter, Porcelli, Rinaldi, and Stoll~\cite{garmatter2022improved}, who show
that for sufficiently small $\varepsilon$ the penalized problem becomes exact.
These continuous relaxations are efficiently solved using matrix-free Newton--Krylov
interior-point methods~\cite{heinkenschloss2014matrix}, enabling high-resolution PDE models.

\paragraph{Time-Dependent Mixed-Integer PDE-Constrained Optimization.}
Time-dependent mixed-integer PDE optimization introduces substantial additional complexity:
the binary variables become functions of time, $z=z(t)$, and the PDE evolves according to
\[
\partial_t u(t) + \mathcal{A}(u(t)) = \mathcal{F}(v(t),z(t)),
\qquad t\in(0,T),
\]
with discrete controls $z(t)\in\{0,1\}^{n_z}$ that must be chosen at each time step.
The combinatorial structure grows exponentially with the number of time intervals, since
a control horizon with $N_T$ time steps yields up to $2^{n_z N_T}$ admissible switching
sequences. Moreover, switching can introduce nonsmooth and discontinuous behavior into the
PDE solution map $t\mapsto u(t)$, especially for hyperbolic systems such as Saint--Venant.

Classical relaxation--rounding and sum-up rounding techniques developed by Sager
\cite{sager2005} and extended in Zeile's dissertation~\cite{zeile2021combinatorial} form the foundation of many modern time-dependent mixed-integer approaches. These methods introduce a relaxed control $\bar{z}(t)\in[0,1]^{n_z}$, solve the relaxed optimality system, and then
construct integer-valued controls $z(t)\in\{0,1\}^{n_z}$ satisfying integral matching
constraints
\[
\int_{t_k}^{t_{k+1}} z_i(t)\,dt \;=\; 
\int_{t_k}^{t_{k+1}} \bar{z}_i(t)\,dt \qquad \forall i,
\]
thereby preserving the structure of the relaxed solution.

Decomposition-based methods for dynamic MIPDECO, introduced by Hahn, Kirches, Manns, Sager,
and Zeile~\cite{hahn2021decomposition}, further reduce complexity by splitting the problem
across space and time. Temporal decomposition replaces the full horizon problem with a
sequence of smaller subproblems, while spatial decomposition separates integer-dependent and
PDE-dependent components. Such techniques greatly improve scalability for high-dimensional,
time-dependent systems, and are particularly effective for hyperbolic dynamics such as
Saint--Venant flow, where repeated simulation of the PDE is required.

Time-parallelization and parareal-inspired decomposition methods developed by Ulbrich~\cite{ulbrich2007generalized} add another layer of structure. By solving coarse and
fine PDE propagations in parallel across multiple time slabs, these methods accelerate the
evaluation of candidate switching sequences, and enable approximate real-time behavior when
embedded inside real-time control frameworks.

Finally, recent developments summarized in Fei, Brady, Larson, Leyffer, and Shen~\cite{fei2023binary} emphasize that dynamic switching control remains computationally demanding, even outside PDE settings. Their observations on the need for relaxations, exact penalties, and structure-preserving rounding techniques extend directly to nonlinear PDE systems. In high-dimensional hyperbolic settings-such as hydropower control-achieving real-time switching performance requires combining these mixed-integer strategies with
surrogate modeling, reduced-order modeling, and decomposition.

\subsection{Complete Stochastic Dynamic Programming Formulation}

Consider the idealized Dynamic Programming Problem under uncertainty for the
operation of a hydropower cascade. Decisions include binary $z(t)\in \{0,1\}$ and continuous $v(t)\in\mathbb{R}^{n_v}$ operation variables for all $t\in [0,T]$. The revenue $\Pi$ depends on the current and future demand together with the generation partially satisfying the demand from other sources, expressed as a random variable $d(\xi(t))$. In general, since future decisions will be taken upon the realization of the uncertainty at that time, we can write them as a policy given the state. 

However, consider that we also have a forecast of the future noise. This comes in the form of a probability density $\tilde{\rho}(\xi(t))$. Naturally, we expect the spread of the density to increase with $t$.

Consider the probability space $(\Xi,\mathcal{B},\mathbb{P})$ with each realization a time-dependent stochastic process $\xi(t)$ over $[0,T]$.
 We will use $\hat{\xi}(t)$ to indicate the live, actual realizations associated with Model Predictive Control.

Generally, the control can be a function of the history of the state, as well as the previous controls up until the current time, as well as the forecasts for the future time:
\[
(z^{\xi}(t),v^{\xi}(t)) = (z(t),v(t))\left(\left\{u(\hat{\xi}(\tau)),z(\hat{\xi}(\tau)),v(\hat{\xi}(\tau))\right\}_{\tau\in[0,t]},\left\{\tilde{\rho}(\xi(\tau))\right\}_{\tau\in(t,T]}\right).
\]
Finally, the final Dynamic Programming problem is defined to be:
\begin{equation}\label{eq:pdeoptdp}
    \begin{array}{rl}
\min\limits_{u,v,z} & \int_{t=0}^T \mathcal{R}_{\xi}\left(-\Pi(u^{\xi},d(\xi(t)))+C(v^{\xi},z^{\xi})\right) dt\\
\text{s.t. } & e(u^{\xi}(x,t)) = f(v^{\xi}(t),z^{\xi}(t)),\, x\in\Omega,\,t\in[0,T]\\ 
& v^{\xi}\in\mathcal{V},\\
& S(u^{\xi})\ge 0,\, \text{almost surely}
    \end{array}
\end{equation}
for $T$ large, theoretically even $T\to\infty$, although in practice the propagating uncertainty would make distant times unrealistic to control for.

Finally, $\mathcal{R}$ refers to a coherent risk measure. We note that the simple choice of an expectation (or risk-neutral stochastic optimization) formally satisfies the discussion of the problem while providing flexibility for incorporating, e.g., Conditional Value at Risk, which would in effect minimize the upper tail of the distribution of negative profit.

To solve~\eqref{eq:pdeoptdp} exactly, one would need to compute the solution of an infinite-dimensional primal-dual stochastic Hamilton-Jacobi-Bellman equation backwards in time given the entire stochastic process up to time $t$, and identify active and inactive measure and time for the state constraint (see, e.g.,~\cite{peng1992stochastic}).

This is entirely intractable in general for nonlinear problems, let alone ones in function space. Thus, we present an Approximate Dynamic Programming approach in which we attempt to integrate multiple techniques across different time scales. This permits simultaneously taking advantage of potential adaptivity to real-time changes, together with taking advantage of potential HPC-assisted offline solution computation to high precision. 

\subsection{Background related to Stochastic PDE Dynamic Programming}\label{sec:backgroundSPDEDP}

This subsection reviews methodological foundations relevant to stochastic dynamic
programming for systems governed by partial differential equations. We summarize
numerical approaches for discretizing uncertainty, algorithmic paradigms for
stochastic PDE control, and techniques for handling probabilistic state constraints.
The discussion emphasizes challenges specific to nonlinear and hyperbolic PDEs,
including limited regularity, discontinuous solutions, and the computational
intractability of classical adjoint-based methods. This background provides the
context for the hierarchical, value-based control framework developed in the
subsequent sections.



\subsubsection{Stochastic Discretization}
\label{subsect_stodis}

In PDE-constrained stochastic optimization, the propagation of uncertainty through
time cannot generally be computed in closed form. Instead, the probability space is
discretized through sampling or approximation, enabling numerical evaluation of
statistical quantities such as expectations and risk measures. The choice of
stochastic discretization depends on the regularity of the parametric solution map
$\xi \mapsto u^\xi$, the dimension of the uncertainty, and the available
computational resources. Here, $\xi$ denotes the stochastic input parameter and $\tilde{\rho}$ denotes the
associated probability measure (or probability density) governing the uncertainty.

\paragraph{Sampling-Based Methods.}
A classical approach is Monte Carlo (MC) sampling, in which independent realizations
$\{\xi^{(i)}\}_{i=1}^M$ are drawn from the probability law $\tilde{\rho}$ and the
corresponding PDEs are solved independently. The expectation of a quantity of
interest $Q(u^\xi)$ is approximated by
\[
\mathbb{E}[Q(u^\xi)] \approx \frac{1}{M}
\sum_{i=1}^M Q\!\left(u^{\xi^{(i)}}\right).
\]
MC estimators are unbiased and dimension-independent, but converge slowly at a rate
$O(M^{-1/2})$. Quasi–Monte Carlo (QMC) methods replace random samples by
low-discrepancy sequences, such as Sobol or Halton points, improving convergence for
sufficiently regular integrands. While straightforward and robust, both MC and QMC
become computationally expensive when high-fidelity PDE solves are required.

\paragraph{Multilevel Monte Carlo.}
To mitigate the computational cost of naive sampling, multilevel Monte Carlo (MLMC)
exploits a hierarchy of discretizations to reduce variance while controlling bias.
Let $u^{\xi,\ell}$ denote the PDE solution at discretization level $\ell$, with
$\ell=0$ corresponding to the coarsest level and $\ell=L$ to the finest.
Using the telescoping identity
\[
\mathbb{E}[Q(u^{\xi,L})]
= \mathbb{E}[Q(u^{\xi,0})]
+ \sum_{\ell=1}^L
  \mathbb{E}\!\left[
    Q(u^{\xi,\ell}) - Q(u^{\xi,\ell-1})
  \right],
\]
MLMC allows coarse levels to be sampled extensively at low cost, while fine levels
are sampled sparsely. Correlated sampling across adjacent levels ensures variance
reduction of the level differences.

\paragraph{Polynomial Approximation Methods.}
Alternatives to sampling include stochastic collocation and polynomial chaos
expansions, which construct polynomial surrogates of the parametric solution map
$\xi \mapsto u^\xi$. These methods can achieve high-order convergence when
$u^\xi$ depends smoothly on~$\xi$. However, for nonlinear hyperbolic PDEs and
switching dynamics, such as those arising in hydropower systems, parametric
regularity is often lost due to shocks or regime changes. In these cases,
polynomial methods may become unstable or inefficient.

\paragraph{Relevance for the Hydropower Control Problem.}
The stochastic inputs in the hydropower cascade induce nonlinear and potentially
discontinuous effects in the forward PDE solutions. For this reason,
sampling-based methods, and in particular MLMC, provide a robust and scalable
approach to uncertainty propagation. In the offline layer, MLMC simulations
generate the statistical data required for surrogate training and value
approximation. The meso and real-time layers subsequently rely on these
discretizations to evaluate stochastic expectations efficiently without imposing
strong regularity assumptions on the underlying dynamics.
\subsubsection{Stochastic PDE Control}

Many applications involve uncertainties, such as uncertain coefficients, unknown boundary conditions, and initial conditions. These problems are commonly formulated as stochastic optimization problems constrained by partial differential equations (PDEs). 

\paragraph{Sample Average Approximation (SAA)}
Sample Average Approximation (SAA) is one of the most widely used approaches for solving such problems numerically. It replaces expectations with empirical averages computed over a set of sampled realizations, resulting in a deterministic approximation that can be solved using standard PDE-constrained optimization techniques.

Depending on the governing PDE, problems can be broadly categorized into elliptic \cite{kouri2013trust,kouri2014inexact,rosseel2012optimal}, parabolic \cite{Borzi2009,Cao2022,Guth2024}, and hyperbolic or parabolic--hyperbolic systems \cite{Esmaili2023}. While elliptic and parabolic problems generally exhibit smoother solution operators, hyperbolic systems introduce additional challenges due to discontinuities, finite-speed propagation, and strong sensitivity to perturbations, which fundamentally affect both analysis and numerical optimization.

\paragraph{Stochastic Approximation (SA)}
Stochastic Approximation (SA) methods, such as stochastic gradient descent, have been proposed as an alternative to SAA \cite{martin2021complexity}. These methods rely solely on first-order information and avoid solving large deterministic systems, but they do not exploit the favorable second-order structure typically present in PDE-constrained optimization. As a result, they tend to converge more slowly than SAA-based approaches. Nevertheless, in settings where SAA becomes computationally prohibitive, SA methods may be useful in the offline phase, trading inexpensive iterations for longer overall runtimes.

Thus far\footnote{By word of mouth in the optimization community.}, hybrid probabilistic models \cite{bandeira2014convergence} have not demonstrated competitive performance relative to SAA or SA methods in stochastic PDE-constrained optimization.

\textbf{Adjoint-based Methods.}
Adjoint-based techniques form the backbone of gradient computations in stochastic PDE control. For each realization $\xi$, one solves the forward PDE
\[
e(u^{\xi}) = f(v,z,\xi),
\]
followed by the corresponding adjoint equation. In the stochastic setting, gradients of expected objectives or risk measures are obtained by aggregating adjoint contributions across samples; see, e.g.,
\cite{Borzi2009,GarreisSuro2021,kouri2013trust,kouri2014inexact,
kouri2016risk,Milz2023,rosseel2012optimal}.  
This framework enables efficient computation of $\nabla_v J$ and $\nabla_z J$ in SAA, SA, and risk-averse formulations.

For elliptic and parabolic PDEs, adjoint equations are well posed and the control-to-state map is sufficiently smooth to permit classical sensitivity analysis. In contrast, for hyperbolic or parabolic--hyperbolic systems, adjoint calculus is significantly more delicate. Even in conservative settings, entropy solutions develop shocks and discontinuities, and perturbations in the control typically induce shifts in these discontinuities rather than smooth variations of the state. Consequently, classical differentiability fails, and more refined concepts such as shift-differentiability and BV-based adjoint formulations must be employed \cite{PfaffUlbrichTLP,Sulb2016,ulbrichpdeopthyper}.

Although these ideas are well developed in deterministic settings, their integration with stochastic sampling remains largely unexplored, with only limited results available for stochastic hyperbolic PDE control \cite{Esmaili2023}. For non-conservative systems, the situation is even less understood, with no generally accepted adjoint framework. This motivates the use of primal and derivative-free approaches in parts of the methodology proposed in this work, where adjoint information may be unreliable or unavailable.

\subsubsection{On Almost Sure State Constraints}

Here the condition 
\[
S(u^{\xi}) \ge 0 
\quad\text{a.s.\ in }\xi
\]
represents a state constraint regarding water level or flow rate. Such constraints appear naturally in hydrodynamic control for several reasons:
\begin{enumerate}
    \item physical feasibility, ensuring the computed free-surface remains above the riverbed,
    \item operational safety, avoiding flooding or excessive discharge that could damage vessels or infrastructure,
    \item environmental protection, guaranteeing ecologically mandated flow and preventing stagnation.
\end{enumerate}

The almost-sure requirement is mathematically challenging because stochastic state constraints define a feasible region
\[
\mathcal{F}
=
\left\{ (v,z) \;:\; S(u^{\xi}(v,z))\ge 0 \ \text{for a.e. }\xi \right\},
\]
which is typically nonconvex and may have an empty interior. Recent developments provide three principal approaches for handling such constraints: relaxation methods, barrier regularization, and KKT multiplier-based formulations. 

\paragraph{Optimality Conditions via Probability-Dependent Multipliers.}
Geiersbach and Wollner~\cite{geiersbach2021optimality} derive first-order necessary conditions 
for optimization problems with almost-sure state constraints in Banach spaces. Their analysis introduces a probability-dependent multiplier $\lambda(\xi)\ge 0$ supported on the active set 
\[
A = \{\,\xi : S(u^{\xi}) = 0\,\}.
\]
The resulting optimality system involves the variational inequality
\[
\mathbb{E}\Big[ 
\langle \nabla_u S(u^{\xi}), \delta u^{\xi} \rangle \,\lambda(\xi)
+ \langle \nabla_{(v,z)} J(u^{\xi},v,z), (\delta v,\delta z) \rangle 
\Big]
\ge 0.
\]
Such multipliers may be singular or highly concentrated, posing severe numerical difficulties in PDE settings.

\paragraph{Relaxation of Almost-Sure Constraints.}
Kouri, Staudigl, and Surowiec~\cite{kouri2023relaxation} propose a relaxation-based surrogate that replaces
\[
S(u^{\xi})\ge 0 
\]
by
\[
S(u^{\xi}) + r_\epsilon(\xi) \ge 0,
\]
where the relaxation term $r_\epsilon(\xi)\to 0$ in probability as $\epsilon\to 0$. The relaxed problems are differentiable, admit stable adjoints, and converge (in a probabilistic sense) to feasible solutions of the original almost-sure constrained problem. This approach is attractive when the number of uncertainty samples is large.

\paragraph{Barrier Regularization.}
Barrier methods~\cite{schiela2009barrier} enforce strict feasibility via the augmented objective
\[
\mathcal{J}_{\beta_B}(v,z)
=
\mathbb{E}[J(u^{\xi},v,z)] 
-
\beta_B\,\mathbb{E}\!\left[\log(S(u^{\xi}))\right],
\qquad 
\beta_B > 0,
\]
ensuring $S(u^{\xi})>0$ for all samples. The barrier term preserves differentiability of the reduced problem and is compatible with adjoint-based gradient computation.

\paragraph{Adjoint Equation Under the Barrier Term.}
For each realization $\xi$, the adjoint $p^{\xi}$ associated with the barrier-regularized functional satisfies
\[
-e^*(p^{\xi})
=
\nabla_u J(u^{\xi},v,z)
-
\beta_B \frac{\nabla_u S(u^{\xi})}{S(u^{\xi})}.
\]
The barrier forcing term grows as $S(u^{\xi})\to 0$, preventing violation of the constraint and stabilizing the optimization.

\paragraph{KKT Multiplier Structure.}
As $\beta_B\to 0$, barrier multipliers converge to the almost-sure multipliers of 
\cite{geiersbach2021optimality}:
\[
\lambda(\xi)
=
\lim_{\beta_B\to 0}
\left(
\frac{\beta_B}{S(u^{\xi})}
\right),
\qquad 
\lambda(\xi) S(u^{\xi}) = 0.
\]
This formal correspondence explains why barrier methods serve as a practical numerical proxy for the theoretically exact but computationally intractable KKT system.

In the hydropower setting, barrier regularization offers the most practical balance between strict feasibility, differentiability, and computational efficiency. Relaxation approaches serve as an effective alternative for large sample sizes, whereas KKT multiplier formulations provide important theoretical insight, but are unsuitable as numerical tools for high-dimensional hyperbolic PDE-constrained stochastic optimization.

\subsubsection{Dynamic and Multistage Stochastic Programming with PDEs}

The paper~\cite{santosuosso2025distributed} considers Stochastic MPC for a Hydropower Cascade, but does not include PDE systems for the water flow as state equations. It presents an efficient distributed scheme for solving the operation of performing MPC with scenarios for the economic operation of the cascade dams.

An important consideration for the operation of MPC is closed-loop stability---that the repeated implementation of the first-time control solution through the operation of MPC yields a trajectory that is physically stable and of approximately optimal performance. We do not attempt to prove closed-loop stability for the system under consideration, but can briefly remark that in the case of hyperbolic systems, the turnpike property, wherein we expect the solution to stay close to an action-minimizing trajectory, is the strategy for establishing stability, e.g.,~\cite{gugat2023turnpike}.

When the statistical integral $\mathcal{R}$ is a risk measure, rather than expectation, risk-consistency becomes a concern. This problem, discussed in~\cite{shapiro2016time,pichler2022risk}, considers the validity of backwards induction in multistage stochastic programming. That is, the risk optimal solution for time stage $t+1$ may prove suboptimal when the risk metric of stages $t$ and $t+1$ taken together, in aggregate, is considered. At the time of this writing, this remains a largely open problem in the case of nonlinear problems.

\subsubsection{Approximate Dynamic Programming of PDE Systems}

The literature on the optimal control of PDEs using Approximate Dynamic Programming (ADP) remains limited. Talaei and co-authors appear to be the primary contributors addressing this class of problems, focusing on ADP-based boundary control for 1D and 2D parabolic equations~\cite{talaei2015adaptive,talaei2016boundary,talaei2017boundary}. The research in \cite{talaei2015adaptive} considers the control of a coupled semi-linear parabolic PDE system with an unknown nonlinearity under Neumann boundary conditions, where the Hamilton-Jacobi-Bellman (HJB) equation is formulated directly in the infinite-dimensional space, with the unknown nonlinearity estimated online by a neural network (NN) approximator. The framework was later extended to multi-dimensional nonlinear PDEs \cite{talaei2016boundary}, specifically for the boundary control of an uncertain 2D Burgers equation. In a subsequent study \cite{talaei2017boundary}, ADP is used to solve the problem with a linear uncertain parabolic PDE, characterized by representing the value function as a surface integral and a novel NN method with a new weight law.

However, the references above concern deterministic systems. ADP for stochastic PDE systems has not yet been
directly applied.  We now mention some representative prior literature on ADP for stochastic ordinary differential equations. \cite{Busic2018} proposed a new perspective for handling decision-making systems with random parameters, characterized by designing an approach to obtain a family of value functions by solving a family of ODEs. In \cite{Poznyak2023}, a differential-neural-network-based min-max robust control method was proposed for nonlinear disturbed systems with uncertain parameters.

\section{On the Control to State Map}

\subsection{State Equation Solution Existence}

A major challenge in hydrological modeling that presents limitations for the formulation and development of optimal control tools is the limited understanding of the basic solution and regularity properties for this class of functions. Namely, it has been generally observed that only a restricted set of initial and boundary conditions is consistent with the existence of solutions that can be stably and accurately simulated numerically. However, precise quantitative requirements are absent, and hydrologists either rely on catalogs of known acceptable conditions, or try various conditions until finding one suitable for simulation, or eschew physical modeling altogether and use empirical ``HBV'' models~\cite{bergstrom2015interpretation}, which, however, generalize poorly across regimes and ambient conditions. 
For a general survey, see, e.g.,~\cite{paniconi2015physically} and also see~\cite{cheng1993tidal} for a well-known representative comprehensive case study of hydrological modeling in the San Francisco Bay.

Formally, the nonconservative form indicates that even Young's measures and statistical solutions are not applicable, and completely Generalized Function Spaces, corresponding to Colombeau algebras, see e.g.,~\cite{colombeau2011elementary}, become the most appropriate functional analysis toolbox.


Letting $\overline{\Omega}=[0,L]\times[0,T]$ for some $L,T>0$, consider defining a hyperbolic system differential operator on $\mathcal{G}\left(\bar{\Omega}\right)$ by:
\begin{equation}\label{eq:diffop}
S_{\epsilon} U_{\epsilon} = \begin{pmatrix}
\frac{\partial A_{\epsilon}}{\partial t}+\frac{\partial Q_{\epsilon}}{\partial x} \\
 \frac{\partial Q_{\epsilon}}{\partial t} +\frac{\partial(Q_{\epsilon}^2/A+g A^2/2)}{\partial x}-\sum\limits_{i=1}^K b^i_{\epsilon} (A_{\epsilon})^{n^A_i}(Q_{\epsilon})^{n^Q_i}.     
\end{pmatrix}
\end{equation}

Let us define a corresponding net of operator $S_{\epsilon}$ acting on $U_{\epsilon}=(A_{\epsilon},Q_{\epsilon})$ with,
\begin{equation}\label{eq:netsaintven}
    S_{\epsilon} U_{\epsilon} = \begin{pmatrix}
\frac{\partial A_{\epsilon}}{\partial t}+\frac{\partial Q_{\epsilon}}{\partial x} \\
 \frac{\partial Q_\epsilon}{\partial t} +\frac{\partial(Q_\epsilon^2/A_\epsilon+g A^2/2)}{\partial x}-g A_{\epsilon}\left(-\frac{\partial z_b}{\partial x}+\frac{n_M^2 Q_{\epsilon} |Q_{\epsilon}|}{A^{10/3}_{\epsilon}}\right).       
    \end{pmatrix}
\end{equation}
Similarly, we must have for all $\epsilon\in (0,\epsilon_0)$, that $\inf \left|A_{\epsilon}(x)\right|\ge C \epsilon^m$ for some $m\in \mathbb{N}$.

Recall Schauder's Fixed Point Theorem (see, e.g.,~\cite[Theorem 2.A and Corollary 2.13]{zeidler1986nonlinear}.

\begin{theorem}\label{th:shauder}

Let $X$ be a Banach space and either
\begin{itemize}
    \item  $M$ be a nonempty, closed, bounded and convex subset of $X$ and suppose $T:M\to M$ is a compact operator, or
    \item  $M$ be a nonempty, compact,  and convex subset of $X$ and suppose $T:M\to M$ is a continuous operator,
\end{itemize}
then $T$ has a fixed point 
\end{theorem}

Application of this Theorem to establish local existence for nonlinear PDEs by considering locally monotone operators is given in~\cite{liu2011existence}.

\paragraph{A Double Loop Fixed Point Scheme}
In~\cite{Oberguggenberger1987}, the existence of global Generalized solutions in the Colombeau algebra $\mathcal{G}(\R)$ and $\mathcal{G}(\R^2)$ is established. Given that our systems of interest are nonlinear, we will consider a fixed-point iteration wherein a solution is first sought with components of the nonlinear PDE fixed. 

Notably, the work in~\cite{Oberguggenberger1987} considers the generic PDE:
\begin{equation}\label{eq:obersemilin}
M_1(x,t)\partial_t u+M_2(x,t) \partial_x u = F(x,t,u)
\end{equation}
and studies existence guarantees under the standing assumption that $\partial_u F$ is bounded for every compact subset of $(x,t)$. It can be noted that this semilinear form does not fit the structure of the nonlinear PDE systems introduced as the focus of this work. This is what necessitates a double-loop fixed-point scheme to derive the Schauder argument. 

In the cases we are considering, we are motivated by applications in design optimization and control with circumstances that necessitate bounds on the solutions. That is, we assume that in accordance with some external necessity, we are only concerned with solutions $u$ satisfying,
\begin{equation}\label{eq:solnbounds}
l_A \le \vert A\vert \le u_A,\,\,l_Q\le \vert Q\vert \le u_Q
\end{equation}
in the pointwise almost sure sense. Thus, the field is both lower and upper bounded. This would correspond to a maximal allowable volume and velocity for fluid dynamics, which is reasonable for container limits and stability, as well as a lower bound, which can be of concern for maintenance of certain environmental and ecological properties. 

For the first example, consider:
\[
\frac{\partial Q}{\partial t} +\frac{\partial(Q^2/A+g A^2/2)}{\partial x}=gA(S_o-S_f(A,Q))
\]
and let $Q^{(0)},A^{(0)}$ be two functions satisfying initial and boundary conditions in \eqref{eq:saintv}. Then, by~\cite{Oberguggenberger1987}, let $Q^{(1)},A^{(1)}$ be the solution to
\[
\frac{\partial Q}{\partial t} +\frac{\partial(Q\frac{Q^{(0)}}{A^{(0)}}+g AA^{(0)}/2)}{\partial x}=gA(S_o-S_f(A^{(0)},Q^{(0)})).
\]
We proceed to consider, next, the solution $Q^{(2)},A^{(2)}$ to
\[
\frac{\partial Q}{\partial t} +\frac{\partial(Q\frac{Q^{(1)}}{A^{(1)}}+g AA^{(1)}/2)}{\partial x}=gA(S_o-S_f(A^{(1)},Q^{(1)})).
\]
and so on. 

From~\eqref{eq:solnbounds}, we know that there exists a $K_l$ and $K_u$, as well as $\tilde{S}_l,\tilde{S}_u$, such that,
\begin{equation}\label{eq:conditionsboundthm}
\begin{array}{l}
K_l\le \frac{l_Q}{u_A}\le \frac{Q^{(0)}}{A^{(0)}} \le \frac{u_Q}{l_A}\le K_u \\
K_l\le \frac{l_Q}{u_A}\le \frac{gA^{(0)}}{2} \le \frac{u_Q}{l_A}\le K_u \\
\tilde{S}_l\le \frac{n_m^2 l^2_Q}{u^{10/3}_A}\le S_f(A^{(0)},Q^{(0)}) = \frac{n_M^2 Q^{(0)}\left\vert Q^{(0)}\right\vert}{(A^{(0)})^{10/3}} \le \frac{n_M^2 u^2_Q}{l^{10/3}_A}\le \tilde{S}_u. \\
\end{array}
\end{equation}

We can obtain the following straightforward Theorem:
\begin{theorem}
    There exists some $\bar{T}$, such that for all $T\le \bar{T}$, there exists a $\bar{B}>0$ such that there exist $K_l,K_u,\bar{S}_l,\bar{S}_u$ such that with initial conditions satisfying~\eqref{eq:conditionsboundthm}, the following bound holds:
    \begin{equation}
    \max\left\{\|A\|_{L_{\infty}},\left\|A^{-1}\right\|_{L_{\infty}},\|A\|_{L_{\infty}},\left\|A^{-1}\right\|_{L_{\infty}}\right\} \le \bar{B}
    \end{equation}
    for time $0\le t\le T$. As such, the map defined by these iterations is a map between two compact spaces within the set of Colombeau Algebras $\mathcal{C}$, and thus a fixed point exists, which by definition is a solution to the PDE. 
\end{theorem}
\begin{proof}
    The bound follows immediately by integration with respect to $t$ and $x$ and general compactness and boundedness arguments. The result follows as a direct application of Schauder's Fixed Point Theorem~\ref{th:shauder}.
\end{proof}
\subsection{Dual and Adjoint Considerations}

Most recently, analysis of optimal control with generalized function spaces has been explored in~\cite{konjik2008foundations} and~\cite{frederico2022calculus}. However, these define the Pontryagin principles and the corresponding Euler-Lagrange equations. In order to develop approximate numerical methods, duality will need to be explored. Duality theory for Colombeau algebras is discussed in~\cite{garetto2005topological}. 

More broadly, a Colombeau Algebra is a Banach Algebra, which is a complete normed Banach space (e.g.,~\cite{soleimani2009optimality}). While the topological dual (with the weak$^*$ topology of convergence) is the set of bounded linear functions $\mathcal{L}(\mathcal{C},\mathbb{R})$, we can recall from distribution theory that this includes the space of compactly supported analytic functions $\mathbb{C}^{\infty}_c$. In fact, by~\cite{garetto2005topological}, this inclusion is a continuous embedding. Although this embedding is not dense, thus making approximation guarantees elusive at the present time, we can consider that a numerical implementation can incorporate multipliers from the space $\mathbb{C}^{\infty}_c$. 

As the space of state solutions $\mathcal{C}$ is a Banach space, we can apply the literature on Optimization in Banach spaces~\cite{hinze2009optimization}.  From this work, we can see that standard arguments for the global convergence properties of the Armijo line search extend from optimization over Hilbert to Banach spaces. 

The closest to our work is~\cite{ulbrich2007generalized}, which explored time-dependent problems and presented a trust region SQP method. The work~\cite{pfaff2015optimal} considers nonlinear hyperbolic conservation laws with switched boundary control, a setting that closely resembles ours. However, in practice, nonconservative shallow water equations are used by hydrologists to model river flow, as dissipative forces, including surface, air, and self-friction, are influential for the river flow dynamics. That is, in many practical river flow models trying to capture rapid variations (hydraulic jumps, rapid transients), practitioners often use non-conservative (also referred to as primitive) formulations of the shallow-water (Saint-Venant in 1d) equations (or mixed/combined formulations), because of their advantages in stability, handling of source terms, computational efficiency or matching observed behavior. See~\cite{abgrall2023new,pudjaprasetya2014momentum,wei2013depth} for a discussion of some of the observations of numerical modeling and simulation of different forms of shallow water CFD systems.

\section{Three-Time Scale Control - Summary}

Notably, the problem described above is far from achieving a practical resolution. There are a number of theoretical and computational hurdles necessary to overcome to approach the problem with some level of clarity and precision. 

In this work, we present a general framework for a three-time-scale approach that we believe constitutes a simultaneously necessary and probably sufficient set of operations to satisfactorily perform the control operation of river flow for hydrological engineering.

One of the most salient problems is the very partial existence results, with strict bounds in the Theorem above, presenting fundamental challenges in ensuring the existence of a reliable control to the state map. Because quantitative guarantees of solution existence may not be available in a number of boundary and initial condition regimes, one may have to spend considerable time offline finding the sets of conditions for which control solutions yield PDEs that exist and can be simulated. 

To this end, the unlimited volume of computational power that can be theoretically applied in offline operation presents the necessity for defining HPC-aided procedures to solve very large control problems and fine PDE discretizations, and attempt to find a catalog of solutions that exist and direct the optimization criteria to favorable directions. From this catalog, a set of global solutions will be refined across different regimes of real-time exogenous noise as it occurs in practice. 

At the meso layer, a set of candidate integer programming (IP) solutions is maintained. Approximate Dynamic Programming will be applied to optimize the continuous variables for problems defined by particles associated with each IP decision set. Heuristics will facilitate the execution of exploration in parallel, as well as potentially try a fresh draw from the solution catalog. 

In real-time, Newton algorithms must be applied to be able to perform computation fast enough for real-time realization and modeling. Using techniques in bifurcation theory to handle solution kernel degeneracy, together with recent methods in Real Time Optimization for Mixed Integer Programs, along with Piecewise Newton Methods, one can define the appropriate toolkit to perform real-time control. The real-time formulations will arise from problems defined by the meso layer serving as a tracking term in a convex functional for the fast MPC problem. Moreover, the models used in real-time are Model-Order Reduced from the more precise PDE simulations in the meso layer.

Integration of the algorithms across the three layers presents a challenging and important mathematical and scientific computing exercise. We summarize the approach below, defining the properties of the three layers and their interfacing.

Recall that, operationally, in the motivating application of dam networks discussed in the Introduction, these components correspond respectively to planning, supervisory control, and real-time actuation tasks. Accordingly, the proposed framework is structured into the following three tiers:
\begin{enumerate}
    \item \textbf{Offline}: This constitutes the running of computational algorithms without consideration of any contemporaneous operation, i.e., any cost-sensible and feasible running of computation. This permits the use of computing clusters with significant HPC resources to run over a long time period. Thus, this setting encourages formulations and discretizations that push the level of dimensionality of the problem to otherwise intractable orders of magnitude. This operation can be performed at the beginning to generate a catalog of a large set of feasible, sensible, stable, and well-performing solutions, with the attempt to arrive at, or get close to, global optimality for some formulations. Subsequently, this catalog will serve as a reliable reference that can feed the operation of meso-layered approximate dynamic programming. Offline solutions can be recommended through AI training and can serve as both potential solution paths to explore or a terminal state and/or value function for Model Predictive Control or Approximate Dynamic Programming, respectively. The offline computation can be periodically run to steadily accumulate more refined and better solutions, while taking feedback from the data observed in real-world operation. Table~\ref{tab:offline} describes the central features of this layer.

\begin{table}[H]
    \centering
    \scalebox{0.95}{\begin{tabular}{|l|p{5cm}|p{4cm}|}
    \hline
\textbf{\underline{Property}} & \textbf{\underline{Description}} & \textbf{\underline{Challenge Addressed}}
    \\
    \hline\hline
\textbf{Role} & Find regions of control-to-state existence & Lack of a priori knowledge \\
\hline
\textbf{Output} & Construct Catalog of Solution Candidates & Intractability for low latency solvers
    \\
    \hline\hline
          \textbf{Hardware} & HPC at Maximal Resource Utilizaiton & Nonlinear High dimensional System \\
       \hline
     \textbf{Latency}  & Unlimited & Extended computation time \\ \hline
     \textbf{Dimensionality} & Millions & Complex system that requires high fidelity for accurate and precise modeling \\
     \hline\hline
     \textbf{Numerical Tools} & CFD numerical solutions & \\ \hline
     \textbf{Data Driven Tools} & Reinforcement Learning & Lack of a priori solution existence guarantees\\ \hline
    \end{tabular}}
    \caption{Operational purpose and resource requirements for Offline Computation}
    \label{tab:offline}
\end{table}

    \item \textbf{Meso-Temporal}: The meso layer represents a reliable computational bridge between the offline and the real-time layer. While not pushing the dimensionality to the extreme degree of the Offline layer, it still retains the full model faithfulness as far as incorporating the PDE formulation. It involves the steady evolution of a moderate number of particles representing potential solutions, and applies procedures that can be solved using Nonlinear Programming software. The offline catalog will feed candidates to the set of particles in a deliberate and informed manner, while the best (appropriately defined) particles will be the source of Model Order Reduction to define the Real-Time operation. The meso layer is meant to take advantage of the comprehensive toolkit available for solving nonlinear constrained optimization problems as commonly used for nonlinear model predictive control. These tools are meant to handle problems with amenable regularity - hence the necessity of the offline catalog to identify appropriate regimes, while also incorporate relatively high dimensional nonlinear models - which would be too difficult to solve in real time latency.
    
\begin{table}[H]
    \centering
    \scalebox{0.95}{\begin{tabular}{|l|p{5cm}|p{4cm}|}
    \hline
\textbf{\underline{Property}} & \textbf{\underline{Description}} & \textbf{\underline{Challenge Addressed}}
    \\
    \hline\hline
\textbf{Role} & Perform reliable nonlinear numerical optimization & Complex nonlinear models \\
\hline
\textbf{Output} & Optimal and near optimal solution estimates for operation & Reliable model predictive control with appropriate forward horizon
    \\
    \hline\hline
          \textbf{Hardware} & High end computing resources devoted entirely for this operation & Leveraging robust reliable numerical optimization solvers \\
       \hline
     \textbf{Latency}  & 0.1-2 hours & Nonlinear high dimensional problem with adaptation to changing environment \\ \hline
     \textbf{Dimensionality} & Tens of Thousands & Balance of fidelity and tractability \\
     \hline\hline
     \textbf{Numerical Tools} & Nonlinear Programming solvers & Reasonable model fidelity for optimal solution generation \\ \hline
     \textbf{Data Driven Tools} & Adaptive tuning and particle management & Mixed binary-continuous problem and potentially disparate regimes\\ \hline
    \end{tabular}}
    \caption{Operational purpose and resource requirements for Meso-scale Computation}
    \label{tab:meso}
\end{table}

    \item \textbf{Real-Time}: The real-time layer corresponds to the actual implemented control in the fast operation of the engineering system. It will use state-of-the-art fast Newton-based solvers, extended recently to problems with mixed-integer decisions. Linearized problems that approximate the current state locally are fed from the meso layer, and the computation is performed locally on safe, fast-performing control code. Feedback from the real evolving state and any additional information will be continuously sent to the meso layer, closing the end-to-end operation of the entire algorithmic structure. 

    \begin{table}[H]
    \centering
    \scalebox{0.95}{\begin{tabular}{|l|p{5cm}|p{4cm}|}
    \hline
\textbf{\underline{Property}} & \textbf{\underline{Description}} & \textbf{\underline{Challenge Addressed}}
    \\
    \hline\hline
\textbf{Role} & Choose real-time adaptive control law & Fast decisions for desired latency for power grid resilience (e.g. ancillary services) \\
\hline
\textbf{Output} & Controls to implement & Physical Operation
    \\
    \hline\hline
          \textbf{Hardware} & Embedded controller & Fast desired latency with proximity to decision system \\
       \hline
     \textbf{Latency}  &  Seconds to Minutes & Rapid adaptation to requirements and new solution estimates \\ \hline
     \textbf{Dimensionality} & Around 100 & Tractability \\
     \hline\hline
     \textbf{Numerical Tools} & Real Time Iteration and MIP solvers & Warm starting and optimizing reduced order models \\ \hline
     \textbf{Data Driven Tools} & Solution and sensor information sent asynchronously up to the Meso hardware & Long run tuning of the architecture \\ \hline
    \end{tabular}}
    \caption{Operational purpose and resource requirements for Real-time Computation}
    \label{tab:meso}
\end{table}

\end{enumerate}

Figure~\ref{fig:3tierschema} presents the overall structure and flow of the computational architecture, together with the communication between the layers.
\begin{center}
\begin{figure}
\includegraphics[scale=0.38]{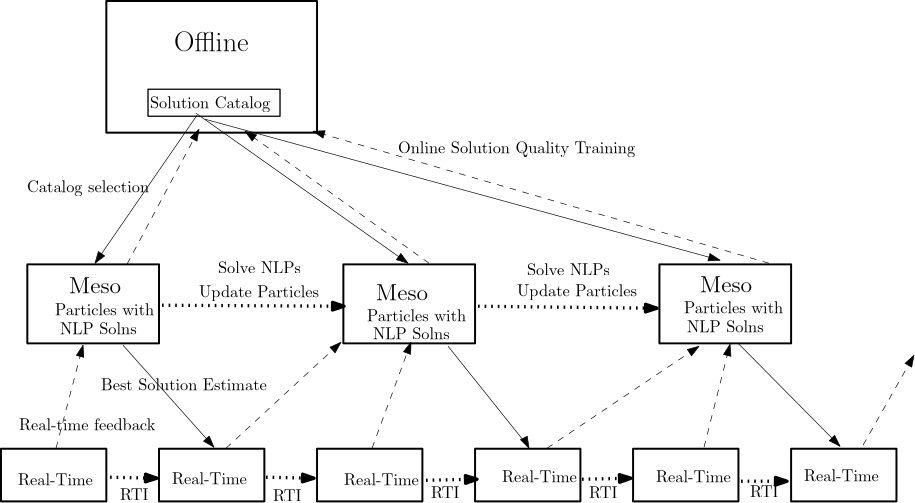}
    \caption{Three time scale computational system architecture for real time control of complex nonlinear dynamics}\label{fig:3tierschema}
\end{figure}
\end{center}
\newpage
\section{Offline Computation of a Catalog of Reference Solutions}

The goal of the offline computation phase is to apply brute force HPC computing power in order to obtain very strong and precise results. In this Section, we describe methods that attempt to solve the Dynamic Programming through a variety of approximating procedures and methods of solution computation. In the context of dam network operation, these problems correspond, for example, to the determination of admissible gate operation policies and flow regulation strategies under hydraulic and operational constraints. We can note immediately that the exact solution of~\eqref{eq:pdeoptdp} by way of defining and solving a Hamilton-Jacobi-Bellman equation is intractable, on account of the infinite dimensionality of the PDE system together with the presence of control constraints. Thus, we present a set of potential methods often used to solve Dynamic Programming approximately that are amenable to discretization and computation.

Before describing the individual components, we summarize the role of each offline method in Table~\ref{tab:offline_methods_summary}. The purpose of Section~5 is not to propose each listed method as a standalone solver for the full stochastic mixed-integer PDE-constrained control problem. Rather, the offline layer combines numerical optimization, reduced switching-structure search, and data-driven learning to construct a catalog of feasible and high-quality reference trajectories. These catalog entries are later used for meso-layer initialization, terminal-value approximation, feasibility guidance, and real-time warm starts.

\begin{table}[H]
\centering
\small
\begin{tabular}{|p{0.25\textwidth}|p{0.30\textwidth}|p{0.14\textwidth}|p{0.18\textwidth}|}
\hline
\textbf{Component} 
& \textbf{Purpose in the offline layer} 
& \textbf{Type} 
& \textbf{Techniques} \\
\hline
Multistage stochastic 
mixed integer optimization 
& Approximate the stochastic dynamic programming problem and generate reference trajectories under uncertainty
& Numerical
& \begin{tabular}{l} 
SAA, \\
(period) ergodic\\
B\&B \\
heurstic search \end{tabular}\\
\hline
Switching-structure reduction
& Replace the full binary trajectory with lower-dimensional switching variables (counts $\underline z$ and (times $\bar z$)
& Numerical
& \begin{tabular}{l} Bilevel \\
and Black-Box\\
Optimization \\
 \end{tabular}\\
\hline
Offline learning
& Learn feasible regions, policy recommendations $\pi_\omega$, and continuous-control bucket classifier $\psi_\omega$ from high-fidelity simulations 
& Data-driven 
& \begin{tabular}{l} 
Reinforcement\\
\& Supervised \\
Learning \\ 
with latent \\ attraction \\ 
domains $\mathcal{D}(v^*)$ 
\end{tabular}\\
\hline
Catalog construction
& Reference 
planning catalog $\mathcal{C}_{\mathrm{off}}$ for for meso-layer warm starts and value approximation
& Hybrid
&  
\begin{tabular}{l}  
Supervised \\
Learning \\
HPC
\end{tabular}\\
\hline
\end{tabular}
\caption{Summary of offline-layer methods and their roles in constructing the reference catalog.}
\label{tab:offline_methods_summary}
\end{table}

\subsection{Formulations}
\subsubsection{Multistage Stochastic Programming}
Multistage Stochastic Programming (MSP) involves finding the solution of an optimization problem over a discrete-time stochastic process for a finite number of steps. One discretizes the stochasticity, typically through scenario generation. For a comprehensive monograph of scenario generation techniques applicable to the primary case study of interest, see, e.g.,~\cite{conejo2010decision}. A scenario is a sequence of samples $\{\xi_{i}(1),\xi_{i}(2),\cdots,\xi_{i}(T)\}$ for the realization of the noise over the $T$ subsequent time steps. Typically, the scenario is an expanding tree, i.e., for $\xi_{i,1}$ one recursively branches to construct subsequent realizations $\{\xi_{i+1}(1),\xi_{i+1}(2),\cdots,\xi_{i+1}(T)\}=\{\xi_{i}(1),\xi_{i+1}(2),\cdots,\xi_{i+1}(T)\}$. Thus, the infamous \emph{curse of dimensionality} applies relative to the time horizon $T$. However, at this offline layer of solution generation, it is expected that the computational load of an HPC server will be pushed to capacity as far as the longest time horizon and most precise stochastic discretization possible. 

Generically, we can construct an empirical sample average approximation to the risk functional $\mathcal{R}$; however, for simplicity of notation, assume that $\mathcal{R}=\mathbb{E}$, permitting the use of a simple sum expression. 

The MSP problem with discretized time and stochasticity now reads as,
\begin{equation}\label{eq:msp}
\begin{array}{rl}
\min\limits_{u,v,z} 
& \frac{1}{NT}\sum\limits_{i=1}^N\sum_{t=0}^T 
\left[-\Pi(u^{\xi_i}(t),d(\xi_i(t)))+C(v^{\xi_i}(t),z^{\xi_i}(t))\right] \\[0.5em]
\text{s.t. } 
& \tilde{e}(u^{\xi_i}(x,t),u^{\xi_i}(x,t+1))
= \tilde{f}(v^{\xi_i}(t),z^{\xi_i}(t)),\; t\in[T],\;\forall i \\ 
& v^{\xi_i}(t)\in\tilde{\mathcal{V}},\; t\in[T],\;\forall i \\
& S(u^{\xi_i}(t))\ge 0,\; t\in[T],\;\forall i \\
& u^{\xi_i}(0) = u_0,\;\forall i \\
& \left\{
\begin{array}{l}
u^{\xi_i}(t)=u^{\xi_j}(t),\;
v^{\xi_i}(t)=v^{\xi_j}(t),\;
z^{\xi_i}(t)=z^{\xi_j}(t),\;\forall t,\, i,j \\
\forall t'<t \text{ s.t. } \xi_i(t')=\xi_j(t')
\end{array}
\right.
\end{array}
\end{equation}
observe that the decisions and state are parametrized by both time and stochastic realization. This models the fact that a decision is made at time $t$. The last constraint, defined as the non-anticipativity constraint, requires non-clarvoyance, i.e., a decision (as well as state) at time $t$ cannot depend on noise realizations after $t$. 

The state constraint is transformed into a constraint that holds for all samples. Of course, this can guarantee only some probability of the almost sure constraint holding. In the meso layer we instead use a barrier function to enforce the state constraint, whereas here, we can consider that the stochastic discretization can be fine enough such that, with a small robustness factor distancing $S(\cdot)$ from the exact break down value of desired/safe operation, we can more or less presume the almost sure state constraint will be satisfied. 

Observe, however, that for the problem to be well-defined, an initial state $u_0$ must be set. Of course, in the general hydropower cascade operation, there is no offline initial state. Thus, one can simply take statistics of the history of the states through observations, and devote resources of computation towards solving~\eqref{eq:msp} as according to the distribution of various states being set as $u_0$.

\subsubsection{Ergodic Optimization}

In the case of ergodic optimization, one considers a stationary distribution that optimizes some long-run average over the history of the states. See, e.g.,~\cite{jenkinson2019ergodic} for a generic exposition. Notably, the exogenous noise in the applications of interest, including the standing example in this paper, can be considered to be multi-periodic. That is, precipitation and electricity demand tend to exhibit cyclical behavior, with periods spanning twenty-four hours, one week, and one year. Approaches towards explicitly considering this process structure are considered in~\cite{kungurtsev2025quasi}. Dynamic Programming, and a computational approach using Stochastic Dual Dynamic Programming, for linear dynamics is considered in~\cite{shapiro2020periodical}. 

In this case, problem~\eqref{eq:pdeoptdp} can be augmented with an additional set of periodicity constraints regarding the probability law of the solution:
\[
\mathcal{P}\left(u^{\xi}(x,t)\right)=\mathcal{P}\left(u^{\xi}(x,t+p)\right),\,\forall t
\]
for a single period $p$. With the multi-periodic, e.g., $\{p_1,p_2,p_3\}$, formalization, one can decompose the solution:
\[
u^{\xi}(x,t) = u^{\xi,1}(x,t)+u^{\xi,2}(x,t)+u^{\xi,3}(x,t)
\]
with
\[
\begin{array}{l}
\mathcal{P}\left(u^{\xi,1}(x,t)\right)=\mathcal{P}\left(u^{\xi,1}(x,t+p_1)\right),\,\mathcal{P}\left(u^{\xi,2}(x,t)\right)=\mathcal{P}\left(u^{\xi,2}(x,t+p_2)\right),\\ \qquad\qquad \mathcal{P}\left(u^{\xi,3}(x,t)\right)=\mathcal{P}\left(u^{\xi,3}(x,t+p_3)\right),\,\forall t.
\end{array}
\]
This formalism permits stochastic discretization techniques incorporating spectral and wavelet bases. 

\subsubsection{Hierarchical Discrete-Continuous Formulation}

In this section, we propose an alternative formulation of the decision variables that leverages the opportunity for HPC-enabled exhaustive computation for finding a global solution, or one that is approximately global. 

In the procedures above, the continuous and integer decision variables are treated to exhibit optionality at every time instant. That is, an operation of a dam could, if it were to be optimal to a criterion, generate frequent alternation between an open and a closed state. Of course, the operation of opening or closing a dam is an engineering task that incurs a cost. In other works on integer-valued time-parametrized control, e.g.,~\cite{leyffer2022sequential}, total variation regularization is used to minimize the switching frequency. For the problems of interest in this paper, we consider that we can evaluate a real economic cost to the operation, and thus the regularization is a real economic quantity, rather than a computational tool. 

We can go one step further, though, as far as considering an inductive bias towards less frequent changes in the binary and integer variables $z(t)$. To this end, we replace this decision variable with an equivalent depiction:
\begin{equation}\label{eq:depictint}
\begin{array}{l}
\underline{z}\in \mathbb{N}^{n_b}
= \left\{\underline{z}_i := \|z_i(t)\|_{TV}
:= \left\vert 
\left\{\underline{t}:\,
\lim\limits_{t\to \underline{t}^-} z(t)\neq 
\lim\limits_{t\to \underline{t}^+} z(t)
\right\}
\right\vert \right\} \\[0.5em]
\bar{z}(\underline{z}) \in \mathbb{R}^{n_b\times \|\underline{z}\|_1}
:= \left\{\bar{z}_{i,j}=\underline{t}:\,
\begin{array}{l}
t>\bar{z}_{i,j-1},\;
\lim\limits_{t\to \underline{t}^-} z(t)\neq 
\lim\limits_{t\to \underline{t}^+} z(t), \\
\bar{z}_{i,0}=0,\, j\le \underline{z}_i.
\end{array}
\right\}
\end{array}
\end{equation}
That is, decompose the problem into 1) the \emph{number of switches} $\underline{z}$ and 2) \emph{how much time to let elapse before a switch} $\bar{z}$. Observe that the dimensionality of the latter quantity depends on the value of the former.

This presents a natural hierarchical optimization procedure, whose formulation we simplify for brevity,
\begin{equation}\label{eq:offtwolayer}
\begin{array}{lr}
\min\limits_{\underline{z}} & \hat{\mathcal{J}}\left(u(\underline{z}),v(\underline{z}),\bar{z}(\underline{z})\right),\\
\text{s.t. } & v(\underline{z})\in\Argmin\limits_{v}\mathcal{J}(u,v,\bar{z}(\underline{z})),
\end{array}
\end{equation}
i.e., find the optimal number of switches for all dams across the total operation horizon, such that the time between switches is chosen in conjunction with the continuous control variables to maximize the total (expected or negative risk) profit. Each evaluation of the integer problem triggers a continuous variate optimization problem. This facilitates the use of off-the-shelf solvers, as constrained continuous optimization tools can be employed as simulation code from a function evaluation. The integer programs, which have a nonlinear black box criterion in the form of this evaluation, can then be solved using the usual branch-and-bound, heuristic, or Bayesian approaches.

\subsection{Numerical Mixed Integer Optimization Methods}

The offline layer has unrestricted access to the full PDE solver, enabling the use of computationally intensive mixed-integer approaches. In this environment, we can run Branch-and-Bound (B\&B), greedy and heuristic search strategies, and Black-Box or Derivative-Free Optimization (BBO/DFO) methods. These approaches explore the discrete decision space associated with dam operating modes and switching configurations. Because offline computation is unconstrained by real-time requirements, the objective is
not merely to find a single optimal solution but to assemble a diverse, high-quality catalog of reference trajectories that span relevant hydrological and operational regimes.

\subsubsection{B\&B}

Classical B\&B dates back to Land and Doig~\cite{Land1960};
see also Mitten~\cite{mitten1970branch} for a general formulation and properties.
B\&B is used in the offline layer as a structured mixed-integer search mechanism,
not as a new standalone algorithmic contribution. Its role is to explore moderate-dimensional
subproblems in which some binary switching variables are fixed or relaxed, and to generate
high-quality incumbent switching patterns for the reference catalog. At a node of the search
tree, a relaxed PDE-constrained subproblem may be solved with $z\in[0,1]$ or with a subset
of binary variables fixed. Nodes whose relaxation lower bound cannot improve the incumbent
are discarded, while promising nodes are further branched.

Because the number of switching configurations grows exponentially with the number of
binary variables and time stages, exhaustive B\&B is not intended for the full stochastic
PDE-constrained control problem. This restricted use is consistent with the MIPDECO
literature: direct discretization of PDE constraints can produce MILPs whose size makes
standard branch-and-bound or branch-and-cut solvers ineffective without additional
structure, preprocessing, or constraint management~\cite{Gnegel2021}. Related work also
emphasizes the use of state elimination and problem-structure reduction for mixed-integer
PDE-constrained optimal control~\cite{Thuenen2022}, as well as lower bounds and convex
relaxations for mixed-integer PDE-constrained problems as ingredients for global optimization
and branch-and-bound-type procedures~\cite{Leyffer2025}. In the proposed architecture,
B\&B is therefore used offline on restricted switching structures, shortened horizons, or reduced discrete subspaces. Even when global optimality is not certified, the procedure can
provide structured candidate trajectories that enrich the offline catalog and seed the
meso-scale particle layer.

\subsubsection{Other Offline Mixed-Integer Strategies}

{\bf Greedy and heuristic search methods} provide simple yet effective strategies for exploring
mixed-integer decision spaces. These approaches iteratively improve a candidate solution
through local modifications \cite{JohnsonPapadimitriouYannakakis1988} or rule-based
adjustments, including hill climbing \cite{RussellNorvig2010}, tabu search
\cite{Glover1986,Glover1989,Glover1990}, and a broad family of metaheuristics
\cite{Talbi2009}. Unlike globally oriented methods such as Branch-and-Bound, greedy and
heuristic schemes prioritize speed and coverage over optimality. In the offline setting,
they complement B\&B by rapidly generating diverse feasible switching configurations through
local search, greedy selection of switching variables, and trajectory-level perturbations
guided by PDE-based cost evaluations. Their primary role is to populate the reference catalog
with structurally varied high-quality solutions, rather than to provide global guarantees,
thereby seeding the meso-scale evolutionary layer with a broad range of integer patterns.

{\bf Black Box Optimization methods} form a broad family of algorithms designed for settings in which a function evaluation constitutes an expensive simulation. Included are Derivative-Free Optimization (DFO) methods for when
gradients are unavailable, unreliable, or prohibitively expensive. Recall that in the case study of interest, a priori reliable numerical adjoint calculation cannot be assured. Thus DFO methods can be used to poll and ultimately assess the existence and regularity of potential reference solutions. Classical DFO
approaches include pattern search \cite{Torczon1997}, mesh-adaptive direct search (MADS)
\cite{AudetDennis2006}, line-search variants \cite{SDBOX,VRBBO}, trust-region models
\cite{Conn2009,MATRS}, and surrogate-assisted schemes \cite{Regis2007}. Black-box and population-based methods such as genetic algorithms \cite{Holland1992},
mixed-variable variants of CMA-ES \cite{Hansen,MADFO}, and particle-swarm optimization
\cite{Kennedy2011} can be useful for exploratory search over nonsmooth or simulator-defined
objective functions when the search variables are reduced or parameterized.

In the present offline PDE-constrained setting, DFO/BBO methods are included as
auxiliary tools for reduced or parameterized search spaces, and also as possible initial
diagnostic tools when the simulator is available only as a black box. They are not intended
to solve the full space--time discretized PDE-constrained mixed-integer control problem
directly. A direct application of generic DFO/BBO to the full control trajectory $(v,z)$
would be computationally inappropriate, even in the offline layer, because the number of
required high-fidelity PDE evaluations would scale poorly with the dimension of the control
trajectory.

Instead, their main role is limited to low-dimensional supervisory or switching-structure
variables arising in the offline catalog-construction phase. For example, after the binary
control has been parameterized by the number of switches $\underline z$ and the switching
times $\bar z$, DFO/BBO methods may be used to explore switching counts, switching-time
patterns, policy parameters, or catalog-selection variables. In this case, each black-box
evaluation may call a high-fidelity PDE simulation or a structured continuous PDE-constrained
optimization subproblem, but the variables optimized by DFO/BBO remain low-dimensional.

A second restricted use case occurs at the beginning of offline exploration, when no
reliable adjoint, differentiability information, relaxation, or exploitable PDE structure is
yet available. In this case, DFO/BBO may be used as an initial black-box diagnostic step to
sample a limited number of candidate controls or parameterized policies, identify feasible
operating regimes, detect solver-failure regions, reveal promising switching structures, and
estimate useful scalings or bounds for later optimization. Once such information has been
obtained, the search should be transferred to reduced, parameterized, or structure-exploiting
PDE-constrained optimization methods rather than continued as an unrestricted black-box
search over the full trajectory.

This restricted use is particularly relevant when adjoints are unavailable, unreliable, or
not meaningful because of nonsmooth switching behavior, wetting/drying transitions, solver
failures, or changes in the numerical PDE solution regime. Thus, DFO/BBO contributes to
catalog diversification, feasibility-region exploration, and early black-box diagnosis, but it
is not used to provide global optimality certificates or to replace structure-exploiting
PDE-constrained optimization methods.

\subsubsection{Comparison of Offline Mixed-Integer Methods for PDE-Constrained Optimization}

The three offline strategies play different roles in constructing the reference catalog.
B\&B is the most systematic approach and is appropriate when the number of active binary
variables is moderate and useful relaxations are available. In such cases, it can provide
high-quality incumbent solutions and, in favorable settings, useful bound information.
Its limitation is the exponential growth of the search tree, especially when binary decisions
are repeated over many time stages or many dams.

Greedy and heuristic search methods have the opposite profile. They are inexpensive and
useful for rapidly producing many feasible switching configurations, but they provide no
optimality certificate and may become trapped in local patterns. Their primary value in the
offline layer is coverage: they generate diverse switching structures and candidate trajectories
that can later be refined, filtered, or used to initialize meso-scale computations.

DFO/BBO methods have two restricted roles. First, they may be used at the beginning
of offline exploration as black-box diagnostic tools when no reliable adjoint,
differentiability information, relaxation, or exploitable PDE structure is yet available.
In this role, they sample a limited number of candidate controls or parameterized policies
to identify feasible operating regimes, solver-failure regions, promising switching
structures, and useful scalings or bounds for later optimization. Second, after the search
space has been reduced or parameterized, they may be used for low-dimensional
switching-time optimization, exploration over the number of switches, policy parameters,
reduced supervisory variables, and catalog-selection variables. They are not intended as
direct solvers over the full space--time discretization of the PDE control problem. In
these restricted roles, DFO/BBO methods can contribute robustness to nonsmooth simulator
behavior and help enrich the catalog in regimes where derivative-based information is
unavailable or unreliable.

Thus, the offline methods are complementary not because they solve the same problem in
different ways, but because they populate the catalog through different mechanisms:
B\&B supplies structured mixed-integer search, greedy and heuristic methods supply rapid
coverage, and DFO/BBO supplies early black-box diagnosis together with reduced-dimensional
exploration of switching structures and simulator-defined regimes.

\subsection{Learning-based Approaches}

This subsection presents the learning mechanisms at the offline layer of the three-tier control architecture. At this level, the complete high-fidelity CFD solver is available and used directly to explore the feasible control space and learn probabilistic and policy structures for the PDE-constrained stochastic optimization problem. No model-order reduction or surrogate approximation is performed here; those tasks belong to the meso and real-time layers discussed later.

\paragraph{Role within the Three-Tier Control Principle.}
The offline layer serves as the computational foundation for the entire learning--control hierarchy. It (i) collects data from full CFD evaluations to characterize feasible and infeasible control configurations; (ii) trains classifiers that predict the existence of a stable numerical PDE solution; (iii) applies reinforcement learning to identify optimal discrete--continuous control combinations using CFD-based feedback; and (iv) estimates domains of attraction for continuous variables, enabling later discretization and policy transfer to the meso scale. In addition, the learned policy $\pi_{\phi^*}$ also serves as a recommender of promising offline catalog entries, guiding the meso layer toward high-value candidate trajectories that the learner has found to perform well under upstream stochastic conditions.

\paragraph{Offline Learning Objectives.}
Given the stochastic PDE system
\[
e(u^{\xi}) = f(v^{\xi},z^{\xi}), \qquad \xi \in \Xi,
\]
where $v$ denotes continuous and $z$ discrete control components, the offline learning process
pursues three principal objectives, described below.

\subsubsection{Predicting Feasibility via Supervised Learning}
The first objective is to construct a classifier
\[
\psi_{\omega}:\ (v,z,\xi)\mapsto [0,1]
\]
that approximates the probability of a stable PDE solution for each control--uncertainty triplet. For every CFD evaluation, a binary label $\chi_i\in\{0,1\}$ records solver convergence or failure. The resulting model $\psi_{\omega}$ provides a feasibility map that restricts later search and prevents infeasible simulations.
\subsubsection{Optimizing Discrete--Continuous Controls via Reinforcement Learning}

By using PDE simulations as the model in a reinforcement learning (RL) loop, we can aim to perform RL with the actions set as the decision variables and the computed objective as the reward. This requires parametrizing the decisions and rewards with deep neural networks, in the case of Deep Reinforcement Learning (DRL).

The second objective formulates the stochastic control problem as an RL environment in which each CFD evaluation corresponds to an episode.
For each feasible configuration $(v,z)$, the observed reward equals the negative risk-adjusted objective
\[
J(v,z)
=\mathbb{E}_{\xi}\!\left[
\mathcal{R}_{\xi}\!\big(-\Pi(u^{\xi}(v,z),d(\xi)) + C(v,z)\big)
\right].
\]
Two RL formulations are supported at the offline layer:
\begin{itemize}
    \item [(a)] \textbf{Bucketed-continuous formulation:} the continuous controls $v$ are discretized into coarse buckets, inducing a finite action set and converting the RL problem to a discrete-action MDP.
    \item [(b) ] \textbf{Discrete-action formulation:} the RL agent outputs only the discrete controls $z$, while the continuous component $v$ is obtained by solving the continuous PDE controlsubproblem with $z$ fixed. This matches the logic of the real-time layer.
\end{itemize}
The policy $\pi_{\phi}$ maps current estimates of the system or uncertainty to control actions $(v,z)$ (or just $z$ in the discrete-action formulation). RL updates improve $\pi_{\phi}$ over repeated high-fidelity evaluations, effectively learning decision rules for the mixed-integer PDE-constrained control problem.
\subsubsection{Estimating Domains of Attraction for Continuous Controls}
For each local minimizer $v^{*}$ obtained by PDE optimization, nearby sampled controls are clustered according to which minimizer they converge to. This yields characteristic regions $\mathcal{D}(v^{*})\subset\mathcal{V}$ representing domains of attraction.

These regions play two crucial roles: (1) they define coarse discrete lattices for continuous-control discretization at the meso scale,
and (2) they provide an optimal bucketing scheme for the RL formulation in which continuous variables are discretized. Each domain of attraction corresponds to a stable, behaviorally coherent bucket in the action space, leading to more effective RL training and more robust policy generalization. Both conceptually and operationally, they can be considered as latent variables that assist the overall solution catalog identification pipeline.

\paragraph{Offline Data Generation.}
A comprehensive dataset is assembled by sampling control configurations and stochastic realizations
\[
\mathcal{D}_{\mathrm{off}}
=\{(v_i,z_i,\xi_i,u_i,J_i,\chi_i)\}_{i=1}^{N_{\mathrm{off}}},
\]
where $\chi_i$ marks CFD feasibility, and $J_i$ is the computed objective. The dataset supports both supervised learning for $\psi_{\omega}$ and policy-gradient or Q-learning updates for $\pi_{\phi}$. Since full CFD runs are used, the offline layer provides a high-fidelity basis for all subsequent surrogate construction.

\paragraph{Offline Reinforcement Procedure.}
Algorithm~\ref{alg:offlineRL} summarizes the offline RL workflow consistent with the paper’s algorithmic structure. In \textbf{(S0$_{\ref{alg:offlineRL}}$)}, the policy $\pi_{\phi_0}$, the feasibility model $\psi_{\omega_0}$, and the dataset $\mathcal{D}_{\mathrm{off}}$ are initialized. Each episode begins with sampling $(v^{(k)},z^{(k)})$ and $\xi^{(k)}$, followed by running a high-fidelity CFD simulation. If the simulation converges, a risk-adjusted objective is computed, and the policy is updated; if it diverges, the feasibility model is updated instead. After $K_{\mathrm{off}}$ episodes, both models are retrained on the accumulated dataset to obtain the final offline policy $\pi_{\phi^{*}}$ and feasibility map $\psi_{\omega^{*}}$.

\begin{algorithm}[H]
\caption{Offline RL with High-Fidelity CFD Evaluations}
\label{alg:offlineRL}
\begin{algorithmic}
\Require CFD solver for $e(u)=f(v,z)$; initial policy parameters $\phi_0$; feasibility model $\psi_{\omega_0}$; learning rates $\eta_{\phi},\eta_{\omega}>0$; number of episodes $K_{\mathrm{off}}$
\Ensure Trained policy $\pi_{\phi^{*}}$ and feasibility predictor $\psi_{\omega^{*}}$
\Statex \textbf{(S0$_{\ref{alg:offlineRL}}$)} Initialize $\pi_{\phi_0}$ and $\psi_{\omega_0}$; set dataset $\mathcal{D}_{\mathrm{off}}=\emptyset$.
\For{$k=0,1,\ldots,K_{\mathrm{off}}-1$}
\Statex \spc \spc \textbf{(S1$_{\ref{alg:offlineRL}}$)} Sample candidate controls $(v^{(k)},z^{(k)})\sim\pi_{\phi_k}$ and stochastic 
\Statex \spc \spc realization $\xi^{(k)}\sim\tilde{\rho}(\xi)$.
\Statex \spc \spc \textbf{(S2$_{\ref{alg:offlineRL}}$)} Run high-fidelity CFD simulation $e(u^{\xi^{(k)}})=f(v^{(k)},z^{(k)})$.
\Statex \spc \spc \textbf{(S3$_{\ref{alg:offlineRL}}$)} Determine feasibility $\chi^{(k)}\in\{0,1\}$.
\If{$\chi^{(k)}=1$}
\Statex \spc \spc \spc \textbf{(S4$_{\ref{alg:offlineRL}}$)} Compute risk-adjusted objective $J^{(k)}$.
\Statex \spc \spc \spc \textbf{(S5$_{\ref{alg:offlineRL}}$)} Update policy: $\phi_{k+1}=\phi_k-\eta_{\phi}\,\nabla_{\phi}J^{(k)}$.
\Statex \spc \spc \spc \textbf{(S6$_{\ref{alg:offlineRL}}$)} Append $(v^{(k)},z^{(k)},\xi^{(k)},J^{(k)},\chi^{(k)})$ to $\mathcal{D}_{\mathrm{off}}$.
\Else
\Statex \spc \spc \spc \textbf{(S7$_{\ref{alg:offlineRL}}$)} Update feasibility predictor.
\EndIf
\EndFor
\Statex \textbf{(S8$_{\ref{alg:offlineRL}}$)} Train final models $\pi_{\phi^{*}}$ and $\psi_{\omega^{*}}$ on $\mathcal{D}_{\mathrm{off}}$.
\end{algorithmic}
\end{algorithm}

\paragraph{Offline Outputs and Interface.}
The offline stage provides: (i) the feasibility map $\psi_{\omega^{*}}$ defining regions of numerical solvability; (ii) a control policy $\pi_{\phi^{*}}$ together with a curated catalog of CFD-validated control solutions that the learner recommends for meso-layer initialization and replacement; and (iii) statistical information on attraction domains in $\mathcal{V}$, which define optimal bucketing schemes for continuous-action discretization in RL and the meso-scale surrogate model. These components form the input knowledge base for surrogate management at the meso layer and feedback control at the real-time layer.


\subsection{Catalog Construction}
The overall purpose of the offline phase of the computational control operation is to establish a catalog of candidate solutions. By leveraging the resources, in time and computation, that could be applied to solve the problems formulated in this Section, the limits of scale, accuracy, and precision could be reached, as far as available methods and hardware. 

Note that the \emph{initial} state and decision variable regime are inputs to all of the problems above, and the algorithms are meant to define the optimal sequence of decisions from that state in some forward horizon. Of course, in real-time operation, there is no initial state; the state continuously changes. And so the practical engineering initiation of Hydropower control must be agnostic with respect to this initial state. 

Thus, the appropriate course of action is to construct the catalog of solutions across a variety of initial states $u(x,t_0)$ and control orientations $\{z(t_0),v(x,t_0)\}$. Their choice can be defined based on data observed in the actual operation of the hydropower cascade. Through sensors measuring water height and velocity in river banks, while defining a database of historical control actions, a large dataset of initial conditions can be defined. 

Then, which initial conditions to use can be chosen through:
\begin{itemize}
    \item Sampling, with importance/Thompson sampling used to target rare events of critical importance to Water Engineering authorities
    \item Lattice grid construction across the discretized dimensionality of solutions, carefully tuning to the fact that this may differ across solution forms. 
\end{itemize}

The database of offline solutions can then serve as a useful tool for investigating particular scenarios of interest during the course of hydrocascade operation. Meanwhile, development can proceed as far as formulating, analyzing, and implementing the methods described below in the meso and real-time layers, while developing tools to integrate the offline solution catalog to be fed to the meso layer. 

\textbf{\emph{The critical importance of the component of using ML to learn the domain of the feasible control-to-state map must be emphasized. Given the lack of a priori knowledge and the lack of the availability of this information for the Shallow Water Equations, finding the appropriate set of control valued in some connected subdomain on which to perform optimization is critical. The use of ML permits navigating around the lack of a priori knowledge in this case, presenting a novel synergy of learning with operations research. }}

Given the nonconvexities in the optimization problem as well as the lack of a priori global regularity guarantees, the catalog of offline solutions will serve as valuable warm starts to the faster, ``meso'' level operation. By biasing warm starts to solutions that are in the interior of a region identified to exhibit a reliable control-to-state map, this provides a landscape over which optimization can take place without encountering unforeseen inoperability.

\subsection{Open Challenges and Future Research}

A number of novel methods have been formulated in this Section - specifically the branching-switching formulation and domain-of-attraction learning. As these are entirely new in the literature, to the best of the authors' understanding, research work formalizing them, studying their theoretical learning, convergence and robustness properties, and validating their performance numerically should be performed before including them into a complete water engineering operational pipeline. Other methods, including learning feasible and regular control-to-state regions of the control domain and the use of mixed-integer optimization tools, are applications of existing methods towards a significantly higher dimensional and nonlinear problem relative to conventional baselines. This stress-testing should be investigated numerically and, quite possibly, require additional algorithmic advances in the techniques recommended. 

\section{Meso Sequential Numerical Optimization}

Nonlinear Model Predictive Control (MPC) (e.g.~\cite{allgower2012nonlinear}) is a lookahead approach towards computationally solving Dynamic Programming by repeatedly solving an optimal control problem over a succeeding $T$ time steps. Sequentially, after the problem is solved, the decision corresponding to the first time step is implemented, then the state is evolved and measured, and the optimal control problem at the next time step with the horizon at the same length, thus extending one more time period, is set up and solved. With stochastic MPC, one performs the same operation; however, by solving a multistage stochastic program at each iteration~\cite{mesbah2016stochastic}. We will consider this as occurring in real meso-time, which is slower than the true real-time operation, but develops a steady stream of control sequences that feeds data into the real-time layer.

The incorporation of Terminal Costs or Terminal Constraints is critical for guaranteeing that an NMPC scheme exhibits stability and suboptimality guarantees with respect to the desired Dynamic Programming solution. To this end, the use of Approximate Dynamic Programming (ADP)~\cite{powell2007approximate} can provide a data-driven approach to define a terminal cost that well approximates the true value function at an ideal optimal policy. With ADP, one maintains an empirical model for the Terminal Cost Value function. In general, it can be difficult to develop a reliable scheme for computing such a function; however, the volume of information available across both the catalog of solutions as well as the real-time operation provides hope for being able to have sufficient data of enough scope to be able to properly inform this function.


\subsection{ADP-MPC Formulation}

Finally, the ADP program becomes:
\begin{equation}\label{eq:pdeoptadp}
    \begin{array}{rl}
\min\limits_{u,b,v,z} & \sum_{t=t_0}^{t_{N_T}} \mathcal{R}_{\xi(t)}\left(-\Pi(u^{\xi}(t),d(\xi(t)))+C(v^{\xi}(t),z^{\xi}(t))\right) \\ & \qquad +\mathcal{R}_{\xi(t_{1:N_{T}})}V\left(u^{\xi}(t_{N_T}), v^{\xi}(t_{N_T}),z^{\xi}(t_{N_T}),d(\xi,t_{N_T}),\cdots d(\xi,t_{N_T-\bar{\tau}_d})\right)\\
& \qquad - \beta_B \mathbb{E}\left[\log S(u^{\xi})\right]
\\
\text{s.t. } & e(u^{\xi}(x,t_{\tau})) = 0,\,x\in\Omega \\ 
&  b^{\xi}_{(l)}(x,t_{\tau})=B(v^{\xi}_{(j(l))},z^{\xi}_{(k(l))},t_{\tau}) ,\, x\in\partial \Omega_{(l)},\,\xi\,a.s.\\ 
& \mathcal{B}\left(u^{\xi}_{(r)}(x,t_{\tau}),u^{\xi}_{(r-1)}(x,t_{\tau}),b^{\xi}_{(l)}(x,t_{\tau})\right)=0,\,x\in \partial \Omega_{(l)},\,\xi\,a.s.  \\
& v^{\xi}\in\mathcal{V},\,\xi\,a.s.,
    \end{array}
\end{equation}
where $\mathcal{B}$ defines a generic (partial Dirichlet/Neumann/Robin) boundary condition operator.

The lookahead Approximate Dynamic Programming~\cite{powell2007approximate}, which can also be described as closed-loop Model Predictive Control (MPC) operation, proceeds by the real succession of the current time $\hat{t}$. At each time instant, the problem~\eqref{eq:pdeoptadp} is formed with $t_0=\hat{t}$, and computationally solved. The control corresponding to the initial time step $(\hat{v},\hat{z})$ is performed, and the system is allowed to evolve to the next time step, at which point this entire procedure is repeated.

\paragraph{Decision Variables}  Recall that generically, the controls include both discrete (open-close) and continuous (dam width/aperture) decisions. In the Meso time layer of computation, however, we consider the application of algorithms for (faster) continuous variate decisions. In particular the focus is to develop procedures that can use off the shelf Nonlinear Programming solvers. To this end, the discrete decisions are eliminated by the following heuristics:
\begin{enumerate}
    \item We maintain a set of solutions, each of which corresponds to specific values for the discrete variables. Specifically we consider a set of particles $\mathcal{P}$ which at time $\hat{t}$ defines a set of binary decisions as fixed from $t_0=\hat{t}$ to $t_{N_T}-1$,
\[
\left\{ z^{*,p}_{t}\right\},\,p\in\mathcal{P},t\in\{t_0,\cdots,t_{N_T}-1\}.
\]
    \item The final, receding horizon adjusted, time instance will have the decision variables defined as probability measures, that is, 
    \[
    z^p_{t_{N_T}}\in \mu(\{0,1\}),
    \]
    yielding a continuous optimization problem for all particles $p$.
    \item The total horizon length will be limited to facilitate maintaining a tractably small collection of optimization problems to solve at each time instant. That is, $N_T$ is not particularly large, less than ten. The cost-to-go is meant to define an estimate for the cost value of the system after time $N_T$. This is by design for the ADP approach for defining $V$ using data from both model simulations and real operation.
\end{enumerate}

From an application perspective, these procedures form the computational backbone supporting supervisory and real-time decision-making in hydraulic systems, where rapid and reliable control updates are required.

\subsection{Sequential QCQP (SQCQP) Method}

In this Section, we describe a Trust Region Filter Composite Step Sequential Quadratically Constrained Quadratic Programming method for solving stochastically discretized ADP problems. This is the most comprehensive novel algorithm presented in the manuscript, meant to circumnavigate the challenges in the proposed application. As it presents a contribution to mathematical programming in its own right, relative to other methods in this paper, we will describe the components with significantly greater detail. 

Because of the challenge of defining Lagrange multipliers described above, we incorporate the Filter technique (e.g.,~\cite{fletcher2006brief}) that obviates the need to track the progress in minimizing a Lagrangian penalty function, as well as use QCQP subproblems to define convex subproblems that incorporate second-order information of both the constraints and the objective function, thus harnessing their potential convergence benefits. Finally, the composite step framework, e.g.,~\cite{heinkenschloss2014matrix} obviates the challenge of potentially infeasible linearized constraints. The composite step Trust Region method first computes a direction that reduces the infeasibility measure, and then seeks to decrease the objective while maintaining the new predicted constraint value. 

A filter used to perform globalization involves maintaining the set:
\begin{equation}\label{eq:filter}
\begin{aligned}
\mathcal{F}:=\Big\{&
(u^{*,(1)},b^{*,(1)},v^{*,(1)},z^{*,(1)},f^{*,(1)},c^{*,(1)},\cdots, \\
&u^{*,(F)},b^{*,(F)},v^{*,(F)},z^{*,(F)},f^{*,(F)},c^{*,(F)})
\Big\}
\end{aligned}
\end{equation}
of points that are Pareto optimal with respect to infeasibility and objective value. 

We assume that there is a discretization of the probability space $(\Xi,\mathcal{B},\mathbb{P})$ denoted by a sample set $\xi_n$ with $n\in[N]$. These can constitute, in the standard case, Sample Average Approximation oriented Monte Carlo samples of $\Xi$ for the entire time interval $[0,T]$. We leave time continuous in the formulation, such that discretizing the optimality conditions of the subproblems can proceed with adaptive time-stepping. We will assume that within the trust region radius $\Delta$, the linear and the quadratic approximation to the PDE exhibit local regularity that permits a primal-dual solution for the subproblems. 

We proceed to now define the operations associated with each major iteration of the primary Algorithm, which we later state formally, discuss implementation details, and then derive some theoretical guarantees regarding convergence. 

At the start of each iteration $k$, consider that we have a current estimate $(u^{(k)}(\{\xi_n\}),b^{(k)}(\{\xi_n\}),v^{(k)}(\{\xi_n\}),z^{(k)}(\{\xi_n\}))$ and a trust region radius $\Delta_k$. We trivially consider that the non-anticipativity constraints are satisfied in the definitions of the estimates, and implicitly assume they are enforced throughout the definitions below.

\subsubsection{Normal and Tangential Step Computation}

Let:
\[
\begin{array}{l}
e^{(k)} := e(u^{(k)}(\xi_n)(x,t),t) \\
\mathcal{B}_{(r)}^{(k)} := \mathcal{B}(u^{(k)}_{(r)}(\xi_n)(x,t),u^{(k)}_{(r-1)}(\xi_n)(x,t),b_{(l)}^{(k)}(\xi_n)) \\
B_{(l)}^{(k)} := B(v_{(l)}^{(k)}(\xi_n)(x,t),z_{(l)}^{(k)}(\xi_n)(x,t){\color{magenta}.t}) \\
S^{(k)} := S(u^{(k)}(\xi_n)(x,t),t) 
\end{array}
\]
and their corresponding directional derivatives $e'^{(k)}d$ as applied to $d$, etc. Here we parametrize the control-to-state map $e$ and the state constraint $S$ by $t$ while considering the dams fixed. This can correspond, in the application of dam control, to changes in the precipitation and vessel traffic affecting the river flow $e$ and operational requirements at different times of day for $S$. 

The normal step $(\tilde{n}^{u,(k)}(\{\xi_n\}),\tilde{n}^{b,(k)}(\{\xi_n\}),\tilde{n}^{v,(k)}(\{\xi_n\}),\tilde{n}^{z,(k)}(\{\xi_n\}))$ is defined as the solution to:
\begin{equation}\label{eq:subptrnormal}
\begin{array}{rl}
\min\limits_{u,v,z} 
& \sum\limits_{n=1}^N\int_{0}^{T}
\left(
\int_{\Omega}
\big(e^{(k)}(\xi_n)
+e'^{(k)}(\xi_n)\tilde{n}^{u,(k)}(\xi_n)\big)^2
\,dx
\right)dt
\\
&\quad
+\sum\limits_{n=1}^N\int_{0}^{T}\int_{\partial\Omega}
\Big(
\mathcal{B}^{(k)}_{(r)}(\xi_n)
+\mathcal{B}'^{(k)}_{(r)}(\xi_n)
\\
&\qquad\qquad
\big(\tilde{n}^{u_{(r)},(k)}(\xi_n),
\tilde{n}^{u_{(r-1)},(k)}(\xi_n),
\tilde{n}_{(l)}^{b,(k)}(\xi_n)\big)
\Big)^2
\,dx\,dt
\\
&\quad
+\sum\limits_{n=1}^N\int_{0}^{T}\int_{\partial\Omega}
\Big(
b^{(k)}_{(l)}(\xi_n)
+\tilde{n}^{b,(k)}_{(l)}(\xi_n)
-B^{(k)}_{(l)}(\xi_n)
\\
&\qquad\qquad
-B'^{(k)}_{(l)}(\xi_n)
\big(
\tilde{n}^{v,(k)}_{(j(l))}(\xi_n),
\tilde{n}^{z,(k)}_{(k(l))}(\xi_n)
\big)
\Big)^2
\,dx\,dt
\\[0.5em]
\text{s.t. } 
& v^{(k)}(\xi_n)+\tilde{n}^{v,(k)}\in\mathcal{V},
\ \forall n
\\
& \left\|
\big(
\tilde{n}^{u,(k)}(\{\xi_n\}),
\tilde{n}^{b,(k)}(\{\xi_n\}),
\tilde{n}^{v,(k)}(\{\xi_n\}),
\tilde{n}^{z,(k)}(\{\xi_n\})
\big)
\right\|
\le \Delta_k.
\end{array}
\end{equation}
Now, in contrast to the classical variant, we not only compute the new candidate values for the constraints $e(u^{(k)}+\tilde{n}^{u,(k)})$, etc. (which are conventionally used for actual to predicted constraint infeasibility reduction comparison for adaptive adjustment of the trust region and step acceptance criteria), but the gradients and the Hessian (or practically, Hessian-vector product operators) $e'(u(k)+\tilde{n}^{u,(k)})[v]$ and $\nabla^2 e(u(k)+\tilde{n}^{u,(k)})[v]$.

Let us define
\begin{align*}
\tilde{e}^{(k)}(\xi_n) &:= e^{\xi_n}(u^{(k)}+\tilde{n}^{u,(k)}), \\ 
\tilde{e}^{',(k)}(\xi_n) &:= e^{',\xi_n}(u^{(k)}+\tilde{n}^{u,(k)}), \\
\tilde{e}^{'',(k)}(\xi_n) &:= e^{'',\xi_n}(u^{(k)}+\tilde{n}^{u,(k)})
\end{align*}
and similarly for the other constraints as well as the objective functions:
\[
\tilde{F}^{(k)}(\xi_n)=\{\tilde{\Pi}^{(k)}(\xi_n),\tilde{C}^{(k)}(\xi_n),\tilde{V}^{(k)}(\xi_n),\tilde{G}^{(k)}(\xi_n)\}
\]
and their derivatives, where $\tilde{G}=-\beta_B\log S(\cdot)$ defines the log barrier term, and the first two terms implicitly sum the quantities across $t\in\{t_0,\cdots, t_{N_T}\}$.

As in the offline optimization problem formulations, we take $\mathcal{R}=\mathbb{E}$ on account of the challenges of risk consistency and the notation complications of SAA as applied to risk measures. See~\cite{kouri2016risk} for a description of risk-averse PDE-constrained problems. Observe that with the regularity of~\eqref{eq:subptrnormal}, this problem can be solved by computing the adjoint, and subsequently perform (potential adaptive) discretization methods. For comprehensive description without excess volume, the continuous time formulation is defined for the normalized step, and the discretized time formulation of the tangential direction subproblem, defined below.

Let $\eta$ be a forcing factor.
The tangential direction $\tilde{t}^{(k)}=\{\tilde t^{u,(k)},\tilde t^{b,(k)},\tilde t^{v,(k)},\tilde t^{z,(k)}\}$ is obtained by solving a QCQP with constraints within new values and minimizing quadratic of objective:
\begin{equation}\label{eq:subptrtangential}
    \begin{array}{rl}
\min\limits_{u,v,z} & \frac{1}{N}\sum\limits_{n=1}^N \left[\tilde{F}^{',(k)}(\xi_n)\tilde{t}^{(k)}+\frac{1}{2}\left(\tilde{t}^{(k)}\right)^T\tilde{F}^{'',(k)}(\xi_n) \tilde{t}^{(k)}\right]
\\
\text{s.t. } & -\eta (e^{(k)}-\tilde{e}^{(k)})\le \tilde{e}^{',(k)} \tilde{t}^{u,(k)}+\frac{1}{2}\left(\tilde{t}^{u,(k)}\right)^T \tilde{e}^{'',(k)} \tilde{t}^{u,(k)} \le \eta (e^{(k)}-\tilde{e}^{(k)}),\,\forall \xi_n \\
& -\eta (\mathcal{B}^{(k)}-\tilde{\mathcal{B}}^{(k)})\le \tilde{\mathcal{B}}^{',(k)} \tilde{t}^{(k)}+\frac{1}{2}\left(\tilde{t}^{(k)}\right)^T \tilde{\mathcal{B}}^{'',(k)} \tilde{t}^{u,(k)} \le \eta (\mathcal{B}^{(k)}-\tilde{\mathcal{B}}^{(k)}),\,\forall \xi_n \\
& 
-\eta (B^{(k)}-\tilde{B}^{(k)}-\tilde{n}^{b,(k)})\le \\ & \qquad \tilde{B}^{',(k)} \tilde{t}^{(u,v,z),(k)}+\frac{1}{2}\left(\tilde{t}^{(u,v,z),(k)}\right)^T B^{'',(k)} \tilde{t}^{(u,v,z),(k)}+\tilde{t}^{b,(k)} \\ &\qquad\qquad\qquad\qquad  \le \eta (B^{(k)}-\tilde{B}^{(k)}-\tilde{n}^{b,(k)}) ,\,\forall \xi_n \\
\\ & v^{(k)}(\xi_n)+\tilde{n}^{v,(k)}+\tilde{t}^{v,(k)}\in\mathcal{V},\, \forall n \\
& \left\|\tilde{t}^{(k)}(\{\xi_n\})\right\|\le \Delta_k.
\end{array}
\end{equation}

\subsubsection{Filter and Trust Region Update}

Algorithm~\ref{alg:mesoSQCQP} summarizes the trust-region, filter-based
Sequential QCQP (SQCQP) procedure used at the meso scale to solve the
stochastically discretized ADP--MPC problem.  All mathematical objects
required by the algorithm---constraint residuals, their first- and
second-order derivatives, quadratic feasibility models, and the
quadratic model of the ADP objective---are defined in the preparatory
formulas preceding the algorithm.  

After computing the normal and tangential steps, the new objective and
constraint values are evaluated via
\[
\left\{
\begin{aligned}
&\frac{1}{N}\sum_{n=1}^N
F\!\left((u^{(k)},b^{(k)},v^{(k)},z^{(k)})
+ \tilde{n}^{(k)}+\tilde{t}^{(k)}\right), \\[4pt]
&\left\{
e\!\left((u^{(k)},b^{(k)},v^{(k)},z^{(k)})
+ \tilde{n}^{(k)}+\tilde{t}^{(k)}\right)
\right\},
\end{aligned}
\right.
\]
together with the corresponding boundary and state-constraint residuals.

To make the filter update precise, we first define the infeasibility
measure.  Let 
\[
x := (u,b,v,z)
\]
denote the full vector of state and control variables at a candidate
iterate, and let
\[
e(x),\qquad 
\mathcal B(x),\qquad
B(x)
\]
be, respectively, the PDE residual, the internal coupling/boundary
residual, and the control-induced boundary residual defined in the ADP--MPC
constraint system.  The state-inequality constraint is
\[
S(u) \ge 0,
\]
and we denote its negative part by
\[
S^{-}(u) := \max\{-S(u),\,0\}.
\]
Using these components, the infeasibility measure is defined by
\[
\phi(x) 
:= 
\|e(x)\|
+ \|\mathcal B(x)\|
+ \|B(x)\|
+ \|S^{-}(x)\|,
\]
where $\|\cdot\|$ denotes any norm consistent with the discretization
(e.g., an $\ell^2$ norm of the discrete residual vector or the
corresponding $L^2$ norm of its continuous representation).

A trial step produces a new point
\[
x_{\mathrm{trial}}
:= x^{(k)} + d^{(k)},
\qquad 
d^{(k)} := \tilde n^{(k)} + \tilde t^{(k)},
\]
with objective value
\[
F_{\mathrm{trial}}
:= F(x_{\mathrm{trial}})
\]
and infeasibility 
\[
\phi_{\mathrm{trial}}
:= \phi(x_{\mathrm{trial}}).
\]
The initial filter $\mathcal F_0$ is initialized with the objective and
infeasibility values of the initial iterate $x^{(0)}$. The filter at iteration $k$ is a set of ordered pairs
\[
\mathcal F_k 
= \{(F^{j},\phi^{j})\}_{j=1}^{m_k},
\]
each representing an objective--infeasibility pair from previously
accepted iterates.  The filter acceptance criterion used in 
(S4$_{\ref{alg:mesoSQCQP}}$) is:
\[
(F_{\mathrm{trial}},\phi_{\mathrm{trial}})
\ \text{is acceptable if for all } (F^{j},\phi^{j})\in \mathcal F_k:
\ \ 
F_{\mathrm{trial}} < F^{j}
\ \ \text{or} \ \
\phi_{\mathrm{trial}} < \phi^{j}.
\]
That is, the trial point must not be Pareto dominated by any entry in
the current filter.

If the trial point is acceptable, the filter is updated by
\[
\mathcal F_{k+1}
:= \bigl(\mathcal F_k \cup \{(F_{\mathrm{trial}},\phi_{\mathrm{trial}})\}\bigr)
\setminus 
\bigl\{(F^{j},\phi^{j})\in\mathcal F_k :
F_{\mathrm{trial}} \le F^{j}\ \text{and}\ 
\phi_{\mathrm{trial}} \le \phi^{j}
\bigr\}.
\]
Thus, any dominated points are removed, and the new point is added.  
If both the feasibility-only step and the composite step satisfy the
filter acceptance test, the algorithm may proceed in two parallel
branches or select one according to a heuristic.

To make the algorithm fully explicit, we recall the standard definitions
that govern the trust-region updates.  The \emph{predicted reduction}
associated with the combined step 
$d^{(k)} := \tilde n^{(k)} + \tilde t^{(k)}$ is the quadratic model
\[
\mathrm{pred}^{(k)}
:= -\Big(
   F'^{(k)} d^{(k)} 
   + \tfrac12\, d^{(k)\top} F''^{(k)} d^{(k)}
  \Big),
\]
while the \emph{actual reduction} is
\[
\mathrm{act}^{(k)}
:= F(x^{(k)}) - F\!\left(x^{(k)} + d^{(k)}\right).
\]
The trust-region ratio
\[
\rho^{(k)}
:= \frac{\mathrm{act}^{(k)}}{\mathrm{pred}^{(k)}}
\]
controls the radius update in
(S5$_{\ref{alg:mesoSQCQP}}$) as
\[
\Delta_{k+1} = 
\begin{cases}
\min\{\gamma_{\mathrm{inc}}\Delta_k,\;\Delta_{\max}\},
& \rho^{(k)} \ge \rho_{\mathrm{good}},\\[6pt]
\max\{\gamma_{\mathrm{dec}}\Delta_k,\;\Delta_{\min}\},
& \rho^{(k)} < \rho_{\mathrm{bad}},\\[6pt]
\Delta_k,
& \text{otherwise}.
\end{cases}
\]
where $0 < \rho_{\mathrm{bad}} < \rho_{\mathrm{good}} < 1$,  
\(0 < \gamma_{\mathrm{dec}} < 1 < \gamma_{\mathrm{inc}}\), and $\Delta_{\max}> \Delta_{\min}\in(0,1)$ are tuning parameters.

The iteration begins with the computation of the \emph{normal step} in
(S1$_{\ref{alg:mesoSQCQP}}$), which reduces constraint infeasibility by
solving a feasibility-oriented QCQP within the trust region.  Once the
feasibility-improving direction is obtained, the \emph{tangential
step} is computed in (S2$_{\ref{alg:mesoSQCQP}}$) by minimizing the
quadratic model of the objective while remaining in the approximate
feasible manifold generated by the normal step.  Here, the normal and tangential steps are defined by
\eqref{eq:subptrnormal} and \eqref{eq:subptrtangential}, respectively.
 The combined step is
tested against the filter in (S4$_{\ref{alg:mesoSQCQP}}$), and the
trust-region radius is updated in (S5$_{\ref{alg:mesoSQCQP}}$) based on
the ratio of actual to predicted reduction.  If insufficient progress
is detected and the radius becomes too small, a feasibility restoration
phase is triggered in (S6$_{\ref{alg:mesoSQCQP}}$), where only normal
steps are taken to recover feasibility.  The resulting framework
combines feasibility enforcement, second-order curvature information,
and global convergence through the filter, yielding a robust SQCQP
solver for the stochastic ADP--MPC problem.

\begin{algorithm}[H]
\caption{Sequential QCQP Filter Trust-Region Method for Meso-Scale ADP-MPC}
\label{alg:mesoSQCQP}
\begin{algorithmic}
\Require Initial iterate $x^{(0)}$, samples $\{\xi_n\}_{n=1}^N$, 
initial radius $\Delta_0$, filter $\mathcal{F}_0$, 
forcing factor $\eta$, trust-region update factors 
$(\gamma_{\mathrm{inc}},\gamma_{\mathrm{dec}})$, 
minimum and maximum radii $\Delta_{\min}$ and $\Delta_{\max}$, and user-prescribed tolerance $\varepsilon$ for constraint infeasibility.
\Ensure Sequence $\{x^{(k)}\}$ approaching a stationary point.

\Statex \textbf{(S0$_{\ref{alg:mesoSQCQP}}$)}  
Initialize $k=0$. 
Evaluate objective and constraint infeasibility at $x^{(0)}$ and insert into $\mathcal{F}_0$.

\While{infeasibility $>\varepsilon$}

    \Statex \spc \textbf{(S1$_{\ref{alg:mesoSQCQP}}$)}  
    Compute the \emph{normal step} $\tilde n^{(k)}$ using the feasibility QCQP.
    
    \Statex  \spc\textbf{(S2$_{\ref{alg:mesoSQCQP}}$)}  
    Compute the \emph{tangential step} $\tilde t^{(k)}$ using the quadratic objective  
    \State  \spc \spc \spc  model and feasibility windows.
    
    \Statex  \spc\textbf{(S3$_{\ref{alg:mesoSQCQP}}$)}  
    Form the trial point  
    $x^{(k)}_{\mathrm{trial}} = x^{(k)} + \tilde n^{(k)} + \tilde t^{(k)}$,  
    evaluate 
    \State \spc \spc  objective and infeasibility, and compute the actual/predicted \State \spc \spc reduction ratio $\rho^{(k)}$.

    \Statex  \spc\textbf{(S4$_{\ref{alg:mesoSQCQP}}$)}  
    Apply the \emph{filter acceptance test}.  
    If acceptable, set $x^{(k+1)} = x^{(k)}_{\mathrm{trial}}$ 
    \State \spc \spc and update $\mathcal{F}_{k+1}$.  
    Otherwise reject the step and keep $x^{(k+1)} = x^{(k)}$.

    \Statex  \spc\textbf{(S5$_{\ref{alg:mesoSQCQP}}$)}  
    Update the trust-region radius $\Delta_{k+1}$ using $\rho^{(k)}$ and 
    the standard \State \spc \spc expansion/shrinkage rules.

    \Statex  \spc \textbf{(S6$_{\ref{alg:mesoSQCQP}}$)}  
    If the step is rejected and $\Delta_{k+1}\le\Delta_{\min}$ for a prescribed
$\Delta_{\min}>0$, 
    \State \spc \spc
    invoke a  \emph{feasibility restoration} step (pure normal-step QCQP).

    \Statex  \spc\textbf{(S7$_{\ref{alg:mesoSQCQP}}$)}  
    Increase $k \leftarrow k+1$.

\EndWhile

\Statex \textbf{(S8$_{\ref{alg:mesoSQCQP}}$)}  
Return final iterate $x^{(k)}$.

\end{algorithmic}
\end{algorithm}

\paragraph{Feasibility Restoration Phase}

A standard feature of Filter methods is to incorporate a secondary backup feasibility restoration phase. This procedure is activated upon a sequence of steps yielding an unsatisfactory decline in the value of infeasibility, while the trust region is decreased to a small size, precluding significant escape from the infeasible region. In this case, a sequence of purely normal steps~\eqref{eq:subptrnormal} is computed, and feasibility is uniquely targeted for the trust region and filter update.

\paragraph{Computing Approximate Subproblem Solutions}

The subproblems are formed in function space with respect to $t$ and $x$. Given their structure, the Robinson Constraint Qualification holds for the two subproblems, and the necessary optimality conditions can be solved to obtain a solution. Due to convexity, there is one unique solution. Since the subproblem is a QCQP, the optimality conditions define a linear PDE. 

However, given the existence of state inequality constraints, even with linear problems, multipliers, which become Borel measures in the case of state constraints~\cite{hinze2009optimization}, again lose regularity. This presents an open research problem that needs to be addressed. Otherwise, one can pursue discretizing the subproblems directly, but this faces the usual issues of mesh-dependence and the physicality of solutions~\cite{schwedes2017mesh}. The standard approach to state constraints with PDE control includes a Moreau-Yosida regularization, that is, an $L^2$ penalty in the objective for the constraints~\cite{hintermuller2010pde}. This, however, adds another important parameter, the corresponding regularization multiplicative factor, that needs to be chosen. This complicates any potential convergence analysis, given multiple co-interacting approximations and discretizations, including the barrier term for the original state constraints and any potential spatial (finite difference/volume/element) discretization of the PDEs, as well as the chosen stochastic discretization in formulating the SAA problem.

\paragraph{Potential Theoretical Guarantees}

The method presented can functionally navigate around the problems associated with the challenge of the absence of well-defined Lagrange multipliers. A filter is a purely primal approach to enforce global convergence, that is, convergence to a stationary point for the original problem from any starting point. 

The use of QCQP permits the use of second-order information for accelerating convergence. This is known to be particularly useful in the case of PDE-constrained optimization, given the structure in the matrices defining the Newton iterations. 

However, a number of technical details need to be handled before convergence can be established with mathematical certainty. At the MPC-ADP level, the estimate of the Terminal cost should be an accurate-enough estimate of a value function in the Dynamic Programming formulation. Stability and Suboptimality would have to be established, possibly by applying turnpike methods~\cite{gugat2023turnpike}. For the optimization algorithm convergence guarantees for QCQP would have to be established, for which little is known in the literature (a recent and useful work is~\cite{gonzalez2025quadratic}). Finally, the barrier parameter considered as both theoretical weak convergence as well as, potentially, a concomitant scheme of adjusting the parameter, should be considered to properly understand the guarantees associated with the proposed algorithm.

\subsection{Open Challenges and Future Research}
This Section introduced a novel numerical optimization algorithm for performing Nonlinear Model Predictive Control within an Approximate Dynamic Programming framework for the meso-scale operation. The method is new in the literature, meant to circumnavigate the challenges of a control problem without a numerically suitable dual space together with uncertainty. As such a number of important details must be developed before it can be implemented and tested at scale. These include theoretical guarantees - including the stability and suboptimality guarantees needed for trustworthy model predictive control operation~\cite{allgower2012nonlinear} and convergence theory for the algorithm itself as far as generating iterates whose asymptotic limits solve the optimization problem. The numerical implementation, as far as the appropriate discretization used to even exhibit approximation properties for functions which can only be guaranteed to lie in a Colombeau Algebra, become the next significant hurdle. Altogether this would constitute a serious research program towards reliable nonlinear optimization tools for solving this class of problems, insofar as a meso-scale intermediary between high-fidelity HPC-guided simulations and real-time operation.

\section{Offline--Meso Integration}


This section develops the meso-scale evolution mechanism that refines the discrete--continuous control trajectories between the offline and meso-scale tiers. 
As emphasized earlier, the meso layer does not perform any mixed-integer or continuous optimization beyond the ADP problem that is solved for each particle. Integer programming (IP/MIP) is reserved for the offline layer, where long computation times are acceptable, and for the real-time layer via reduced-order, 
linearized models. The meso layer instead operates through a structured evolutionary procedure that 
acts on a population of discrete trajectories.

In practical water management systems, groundwater models are often coupled with surface hydrodynamics through exchange fluxes at riverbeds, reservoirs, or recharge zones. These coupled surface–subsurface systems require mass-conservative interface conditions and may be solved using either partitioned or monolithic strategies. In turn, this must take into account uncertainty in the exogenous evolution in quantities inflows, boundary conditions, and system response, and their impact on the selection and robustness of control actions.

\paragraph{Offline Catalog as Terminal Value Surrogate.}
The offline catalog $\mathcal{C}_{\mathrm{off}}$ consists of high-quality, 
PDE-validated control trajectories collected through Branch-and-Bound, greedy 
heuristics, DFO/BBO methods, and offline reinforcement learning.  
Each trajectory includes the full continuous-discrete control sequence, the 
associated PDE state evolution, and its realized cumulative cost.  
Thus $\mathcal{C}_{\mathrm{off}}$ provides a data-driven approximation of the
downstream cost-to-go that appears in the ADP terminal term of
\eqref{eq:pdeoptadp},
\[
V\!\left(u^{\xi}(t_{N_T}), v^{\xi}(t_{N_T}), z^{\xi}(t_{N_T}),
         d(\xi,t_{N_T}), \ldots\right).
\]

Let 
\[
x := \big(u^{\xi}(t_{N_T}), v^{\xi}(t_{N_T}), z^{\xi}(t_{N_T}),
         d(\xi,t_{N_T}),\ldots\big)
\]
denote the terminal augmented state.  
For each catalog entry $c\in\mathcal{C}_{\mathrm{off}}$, let 
$x^{(c)}_{\mathrm{terminal}}$ be its terminal state and 
$J^{(c)}_{\mathrm{tail}}$ its realized continuation cost.  
A nonparametric ADP terminal-value approximation takes the form
\[
V_{\mathrm{ADP}}(x)
\approx 
\min_{c\in\mathcal{C}_{\mathrm{off}}}
\Big(J^{(c)}_{\mathrm{tail}}
     +\omega(\|x-x^{(c)}_{\mathrm{terminal}}\|)\Big),
\]
where $\omega(\cdot)$ is a stabilization weight.  
Alternatively, a parametric surrogate
\[
V_{\theta}(x)=\mathrm{NN}_{\theta}(x)
\]
is trained on the terminal pairs 
$\big(x^{(c)}_{\mathrm{terminal}}, J^{(c)}_{\mathrm{tail}}\big)$.  
Thus, the offline catalog directly informs the ADP terminal cost, providing a 
computationally efficient proxy for the value function and stabilizing the 
rolling-horizon MPC behavior at the meso layer.

\paragraph{Particle Representation.}
Each particle corresponds to a mixed control trajectory for the next $T$ time 
steps.  The binary variables are fixed for the first $T-1$ steps and relaxed in the 
final step:
\[
z^{(p)}(t)\in\{0,1\},\quad t=\hat t_c,\ldots,\hat t_c+T-2,
\qquad 
z^{(p)}(\hat t_c+T-1)\in[0,1].
\]
This preserves discrete feasibility while allowing a low-dimensional continuous 
relaxation for the ADP solver.  
No additional optimization of $v$ or $z$ is performed at the meso scale.

\paragraph{Evolution Mechanism.}
Each particle $p$ is independently evaluated by solving its ADP subproblem, 
yielding a value estimate $J^{(p)}$.  
The meso evolution then proceeds as:
\begin{enumerate}
    \item[(i)] \textbf{ADP evaluation:} compute $J^{(p)}$ for each particle.
    \item[(ii)] \textbf{Selection:} rank particles by $J^{(p)}$ and identify underperformers.
    \item[(iii)] \textbf{Replacement:} inject high-quality trajectories from $\mathcal{C}_{\mathrm{off}}$.
    \item[(iv)] \textbf{Mutation:} flip 1--2 binary entries of $z^{(p)}$ 
              following the refined integer uniform sequence~\cite{MATRS}.
    \item[(v)] \textbf{Retention:} keep the best particles unchanged.
\end{enumerate}
This purely evolutionary mechanism maintains population diversity without invoking 
expensive mixed-integer or continuous optimization.

\paragraph{Meso-Layer Master Algorithm.}
Algorithm~\ref{alg:mesoparticlesrefined} summarizes the evolutionary procedure used to
propagate and refine the population of discrete--continuous control trajectories at the
meso scale.  
The algorithm operates on a fixed prediction window of length $T$ and treats each particle
as an independent candidate policy that is evaluated solely through its ADP value
$J^{(p)}$.  
Because the meso layer performs no discrete optimization of its own, the ADP solver
functions as the only source of objective information, and all population updates are driven
by the relative ordering of the ADP values.

Step (S0$_{\ref{alg:mesoparticlesrefined}}$) initializes the particle set either from the
offline catalog $\mathcal{C}_{\mathrm{off}}$ or by sampling from the learned policy
$\pi_{\phi^*}$, thereby ensuring that the initial population spans both historically optimal
patterns and recently learned, policy-guided structures.  
This initialization phase is critical: it seeds the particle ensemble with high-quality
discrete configurations that reflect the structure of the true dynamic programming solution.

In each iteration, the algorithm performs three core operations.  
First, step (S1$_{\ref{alg:mesoparticlesrefined}}$) performs a full ADP evaluation for every
particle. This means solving the reduced ADP optimization problem with the particle's
fixed discrete schedule $\{z^{(p)}(t)\}$, obtaining a continuous control refinement
$\{v^{(p)}(t)\}$ and a numerical value $J^{(p)}$.  
Because all particles can be evaluated in parallel, this stage constitutes the main
computational workload but scales efficiently on modern computing architectures.

Second, step (S2$_{\ref{alg:mesoparticlesrefined}}$) ranks particles according to their ADP
values. This creates a strictly ordered performance hierarchy that drives every subsequent
evolutionary action.  
Underperformers are flagged for replacement or mutation, while the top-ranked particles form
a stable “elite set” that anchors the distribution of the population.

Third, steps (S3$_{\ref{alg:mesoparticlesrefined}}$)-(S5$_{\ref{alg:mesoparticlesrefined}}$)
apply the evolutionary operators:
replacement injects high-quality discrete schedules from the offline catalog,
mutation introduces local variation by flipping one or two components of $z^{(p)}$, and
retention preserves the best-performing particles.  
This yields a balance between exploitation (preserving and refining strong discrete
structures) and exploration (introducing diversity).  
Mutations follow the refined integer uniform sequence~\cite{MATRS}, ensuring that variation
is spatially well-distributed across the binary vector rather than randomly clustered.

Step (S6$_{\ref{alg:mesoparticlesrefined}}$) performs periodic CFD validation of the best
particles.  
If the surrogate $\hat{\mathcal{S}}_{\theta}$ drifts significantly from the high-fidelity
PDE solver, the surrogate is retrained using the discrepant CFD snapshots.  
This maintains consistency between the meso layer’s ADP evaluations and the true PDE
dynamics without imposing CFD costs at every iteration.

Finally, after $K_{\mathrm{meso}}$ iterations, step (S7$_{\ref{alg:mesoparticlesrefined}}$)
returns the refined population and the single particle with the best value estimate
$J^{(p)}$.  
This best-performing discrete-continuous control pair
$(\hat v(\hat t_c),\hat z(\hat t_c))$ is passed directly to the real-time layer, where it is
linearized and used as the initial guess for rapid model-predictive optimization.

Overall, the algorithm functions as a dynamic bridge between the offline catalog and the
real-time MPC solver: it propagates discrete structural information forward in time,
injects offline knowledge when necessary, maintains population diversity through mutation,
and ensures surrogate fidelity through selective PDE validation.  
These properties make the meso layer a robust intermediate stage capable of tracking,
adapting, and refining discrete hydropower operating modes in fluctuating stochastic
regimes.

\begin{algorithm}[H]
	\caption{Meso-Scale Particle Evolution}
	\label{alg:mesoparticlesrefined}
	\begin{algorithmic}
		\Require Particle set $\mathcal{P}_{\hat t_c}$; ADP solver; 
		         offline catalog $\mathcal{C}_{\mathrm{off}}$; 
		         mutation rate $\lambda$; $K_{\mathrm{meso}}$
		\Ensure Updated population $\mathcal{P}_{\hat t_c+K_{\mathrm{meso}}}$ 
		        and the best control pair
		
		\Statex \textbf{(S0$_{\ref{alg:mesoparticlesrefined}}$)} 
		         Initialize $\mathcal{P}_{\hat t_c}$ from $\mathcal{C}_{\mathrm{off}}$ 
		         or policy $\pi_{\phi^*}$.
		
		\For{$k=0,1,\ldots,K_{\mathrm{meso}}-1$}
		
		    \For{each particle $p\in\mathcal{P}_{\hat t_c+k}$}
		        \Statex \spc\spc \textbf{(S1$_{\ref{alg:mesoparticlesrefined}}$)} 
		        Solve particle's ADP problem to obtain $J^{(p)}$.
		    \EndFor
		
		    \Statex \spc \textbf{(S2$_{\ref{alg:mesoparticlesrefined}}$)} 
		    Rank particles; identify underperformers.
		
		    \Statex \spc \textbf{(S3$_{\ref{alg:mesoparticlesrefined}}$)} 
		    Replace worst-performing particles with samples from $\mathcal{C}_{\mathrm{off}}$.
		
		    \Statex \spc \textbf{(S4$_{\ref{alg:mesoparticlesrefined}}$)} 
		    Mutate selected particles:
            \Statex ~ \spc \spc \spc \spc  flip 1--2 entries of $z^{(p)}$ 
                    following the uniform integer sequence.
		
		    \Statex \spc \textbf{(S5$_{\ref{alg:mesoparticlesrefined}}$)} 
		    Retain top-performing particles unchanged.
		
		    \Statex \spc \textbf{(S6$_{\ref{alg:mesoparticlesrefined}}$)} 
		    Periodically validate top particles with CFD and update surrogate.
		
		\EndFor
		
		\Statex \textbf{(S7$_{\ref{alg:mesoparticlesrefined}}$)} 
		Return updated population and best-performing particle.
	\end{algorithmic}
\end{algorithm}

\paragraph{Discussion on Mutation and Surrogate Feedback.}
No gradients or adjoint sensitivities are used at the meso-scale particle evolution master algorithm scale; the only signal guiding 
evolution is the ADP value estimate $J^{(p)}$.  
Binary mutation uses the refined integer sequence~\cite{MATRS} to ensure 
non-clustered, well-distributed flips.  
Solving the ADP problem is performed by applying off-the-shelf Nonlinear Programming solvers in parallel for all particles $p$. The meso layer then passes the best control pair $(\hat v(\hat t_c),\hat z(\hat t_c))$ to the real-time layer for linearization and execution. Meanwhile the information gained from the ADP solution computation is added in a database for learning  at both the Meso layer (for the terminal cost $V$) and the offline layer. 

\subsection{Open Challenges and Future Research}

There are a number of details to ensure the operation described in this Section is succesful. Notably, the optimal, or possibly satisficing set of optimal solution must change at a sufficiently slow latency, relative to the computationally required load. Conversely, the number of particles must evolve, through mutation to new solutions and selection of the leading solution to feed to the lower layer, with sufficiently frequent latency for the solution pathfollowing to proceed adequately.

\section{Fast Newton Solutions of Reduced Order Models at Real Time}

Ultimately, the practical goal of the problem(s) discussed in this paper is to perform operations in real-time, that is, sequential decisions that make changes within a short time interval. The shorter the time interval, the more quickly the system can adapt to changes in the environment. At the same time, short time intervals prevent the computation of complex and sophisticated procedures. In a real-time control context, coupled complex PDE systems must typically be approximated using surrogate models or reduced-order interface representations in order to maintain computational efficiency within operational time constraints.

In practical terms, for dam networks, this corresponds to constructing and selecting from a set of admissible control policies that encode feasible gate operations under varying hydrological conditions. These models are integrated with the meso layer by Model Order Reduction (MOR). MOR, as described in the well-known three-volume text~\cite{benner2017model}, is a comprehensive toolkit for obtaining simpler, e.g., linear models that approximate a more complicated nonlinear model with some degree of instrumental accuracy. The specific details of performing MOR for meso-to-real-time optimization would constitute a comprehensive research program, and in this paper, some general suggestions and options are presented, without claiming authority on their optimality.

\subsection{Real Time Control}

The implementation of MPC in real-time, that is, with short durations between measuring changes in the environment and performing the next control action, depends on fast computational procedures that take advantage of the fact that drastic changes in the environment typically only occur over long time intervals. A principal component is the preferential application of Newton methods. Notably, Newton iterations are quadratically convergent to a solution from starting points sufficiently close, and so if a method traces a trajectory of solutions while maintaining proximity within the radius of fast convergence, real-time operation can be truly fast~\cite{deuflhard2011newton}.

Within engineering systems applying real-time MPC, there are two standard paradigms: Real Time Iteration (RTI)~\cite{diehl2005real} and advanced step NMPC~\cite{huang2009advanced}. With RTI, a sequence of Newton steps is used directly to perform path following along the sequence of Nonlinear Programming problems defining the control for the receding horizon. The perturbation, beyond the increment in time, thus dropping one stage and adding another, includes state-model mismatch, i.e., the measured state replaces the initial state. With advanced-step MPC, 1) after a control is implemented, a new optimization problem is solved, using the previous model; 2) once the problem is solved, the real state is measured, and a corrector Newton step updates the next control. 

To account for active set changes during the procedure, predictor-corrector quadratic programs (QPs) are used to traverse these while maintaining Newton fast convergence. See, e.g.,~\cite{kungurtsev2014sequential,kungurtsev2017predictor}. Meanwhile, Filtering and Moving Horizon Estimation are optimization techniques that enable the simultaneous estimation, from real-time observations and measurements, of parameters and states~\cite{rawlings2006particle}.  

\subsection{Recent Advances in Mixed Integer Real Time Control}

Real-time MPC has classically been limited to continuous optimization. Integer and other combinatorial decision variables are fundamentally distinct from the smooth trajectories - states and parameters changing gradually over time - amenable to Newton methods. Instead, a change in the solution involves a discrete jump, for which, generically, fast computational methods become difficult to develop with demonstrable speed and robustness~\cite{guddat1990parametric}.   

Recently, the study of mixed integer optimization within the context of NMPC has begun to be addressed. Advances in the power of computational hardware have improved integer programming solvers to the point that problems of a dimension that is interesting and relevant to many industrial applications become within range for real-time numerical optimization-based control. For an early work in this space, see~\cite{kirches2011fast}. Extensions of RTI for mixed integer problems are explicitly developed in~\cite{de2020real,quirynen2021sequential}. 

Notably, fast computation becomes reasonable when a significant subset of the discrete decision variables is fixed, leaving a small number remaining for optimization (which can be found through a slower-time branch and bound procedure, in the references above, or through the meso-to-real-time operative integration, here). If $\mathcal{D}\subset [n_z]$ are the variables worth considering, and $\bar{e}$ is the linearized model, and $F_k$ is the total objective (i.e., including the future time over the fixed horizon), two available approaches are available.

One can solve a Mixed Integer Linear Program, wherein the discrete variable appears explicitly in the state model, which is otherwise linearized with respect to the continuous variables:
\[
\begin{array}{l}
\min\limits_{d_u,d_v,z_{\mathcal{D}}}\, \nabla_{u} F_k^T d_u
+\nabla_{v} F_k^T d_v \\
\text{s.t.   } u_k(t+1)+d_u(t+1) = \bar e_k(z)+\nabla_u\bar{f}_k (z,t)^T d_u+\nabla_u\bar{f}_k (z,t)^T d_v\\
\qquad\qquad\qquad -\Delta_k \le d_v \le \Delta_k, \,\,z_{\mathcal{D}}\in \mathbf{P}(\mathcal{D})
\end{array}
\]
where $\mathbf{P}(\mathcal{D})$ is the power set (set of all subsets) of $\mathcal{D}$.

In~\cite{quirynen2021sequential}, a mixed integer quadratic optimization problem is presented, with a similar two-time-scale approach with sequential linearization of the underlying linear model. The paper presents an approach and software for solving this set of problems, that is, of the form:
\[
\begin{array}{l}
\min\limits_{d_u,d_v,z_{\mathcal{D}}}\, \nabla_{u} F_k^T d_u
+\nabla_{v} F_k^T d_v +\frac{1}{2} \begin{pmatrix} d_u \\ d_v \end{pmatrix}^T H_k \begin{pmatrix} d_u \\ d_v \end{pmatrix}  \\
\text{s.t.   } u_k(t+1)+d_u(t+1) = \bar e_k(z)+\nabla_u\bar{f}_k (z,t)^T d_u++\nabla_u\bar{f}_k (z,t)^T d_v,
\end{array}
\]
where $H_k$ is an approximation to the Hessian of the Lagrangian with respect to the continuous variables. An efficient numerical procedure that can solve these in real time on embedded (that is, restricted computational capacity) hardware.

For a recent warm start procedure for considering stochasticity within the context of NMPC with MILPs, see~\cite{markhorst2024two}. 

\subsection{Piecewise Linear Approximation of the Shallow Water Equations} \label{sec:pwLinearApproxSWE}

A prominent approach to obtaining tractable approximations of the nonlinear Shallow Water equations relies on piecewise linear representations of the governing dynamics arising from semi-implicit space-time discretizations. In conservative form, the two-dimensional Shallow Water system reads
\begin{equation}
\begin{aligned}
\partial_t\eta + \nabla \cdot (\eta\mathbf{u}) &= 0, \\
\partial_t (\eta\mathbf{u}) + \nabla \cdot (\eta\mathbf{u} \otimes \mathbf{u}) + g\eta\nabla H &= \mathbf{S}(\eta,\mathbf{u}),
\end{aligned}
\label{eq:SWE_continuous}
\end{equation}
where $\eta$ denotes the water depth, $\mathbf{u}$ the depth-averaged velocity, $H = h + z_b$ the free-surface elevation relative to the bathymetry $z_b$, $g$ the gravitational acceleration, and $\mathbf{S}$ collects source terms such as bottom friction or Coriolis effects.

Within finite-volume or finite-difference frameworks, semi-implicit discretizations treat the gravity-driven pressure term implicitly while retaining an explicit or linearized treatment of the advective fluxes. Denoting by $\eta_i^{n}$ and $\mathbf{u}_i^{n}$ the cell-averaged unknowns at time level $t^n$, a typical semi-implicit update can be written schematically as
\begin{equation}
\begin{aligned}
\frac{\eta_i^{n+1} -\eta_i^{n}}{\Delta t}
+ \sum_{j \in \mathcal{N}(i)} F_{ij}^{n} &= 0, \\
\frac{(\eta\mathbf{u})_i^{n+1} - (\eta\mathbf{u})_i^{n}}{\Delta t}
+ \sum_{j \in \mathcal{N}(i)} \widehat{F}_{ij}^{n}
+ g\eta_i^{n+1} \nabla H_i^{n+1} &= \mathbf{S}_i^{n+1},
\end{aligned}
\label{eq:SWE_semiimplicit}
\end{equation}
where $F_{ij}^{n}$ and $\widehat{F}_{ij}^{n}$ denote numerical fluxes across cell interfaces and $\mathcal{N}(i)$ is the set of neighboring cells.

The key observation underlying the piecewise linear formulation is that nonlinearities associated with wetting-and-drying transitions, positivity constraints on $h$, and friction laws can be expressed through max-min operators or complementarity conditions. For instance, enforcing non-negativity (or other, more robust bound) of the water depth leads to relations of the form
\begin{equation}
H_i^{n+1} = \max\!\left\{0, \eta_i^{n+1}+z_b\right\}.
\label{eq:positivity}
\end{equation}
As a result, the fully discrete problem at each time step can be cast as a piecewise linear system
\begin{equation}
A(\boldsymbol{\sigma}) \, \mathbf{x} = \mathbf{b}(\boldsymbol{\sigma}),
\label{eq:PL_system}
\end{equation}
where $\mathbf{x}$ collects the discrete water levels and velocities, $\boldsymbol{\sigma}$ encodes the active set associated with wet and dry cells, and the matrices $A(\boldsymbol{\sigma})$ and vectors $\mathbf{b}(\boldsymbol{\sigma})$ depend linearly on this regime selection. 

The combination of equations across $\delta t$-iterative time steps from some initial $t=0$ to some final time $T=H\delta t$, together with any piece-wise kinks in the state, provide discretization knots. The system will be interpreted as exhibiting a linear interpolation between these knots, forming an optimization problem with many mixed integer linear constraints. 

For such systems, specialized iterative algorithms based on nested Newton or active-set strategies admit finite termination and strong robustness properties \cite{BrugnanoCasulli2008,BrugnanoCasulli2009}. Crucially, these formulations preserve essential structural features of the continuous equations, including exact mass conservation at the discrete level, unconditional stability with respect to the time step size, and a consistent and reliable treatment of wetting-and-drying interfaces, while substantially reducing the computational burden relative to fully nonlinear implicit schemes.

High-resolution extensions of these ideas to multidimensional free-surface hydrodynamics further demonstrate that accurate advection-diffusion balances can be maintained without sacrificing stability or monotonicity \cite{CASULLIZANOLLI2007}. In particular, the resulting schemes can be interpreted as regime-dependent linearizations embedded within a globally conservative finite-volume framework. More generally, nested Newton-type strategies for solving the associated nonsmooth algebraic problems provide a systematic mechanism for organizing such piecewise linear models within large-scale simulations \cite{CasulliZanolli2010}.

From the perspective of optimization under PDE constraints, the induced piecewise linear or piecewise affine structure of the discrete dynamics is particularly appealing, as it naturally aligns with mixed-integer, complementarity-based, and hybrid optimization formulations while retaining a close connection to the underlying physical processes. At the same time, the inherent nonsmoothness and switching behavior introduced by these approximations highlight fundamental analytical and computational challenges—such as limited regularity of control-to-state mappings, sensitivity to active-set changes, and the scalability of large-scale solvers—that motivate the broader discussion developed in this work.

\subsection{Open Challenges and Future Research}
Even with the Model Order Reduction discussed below, the dimensionality of the optimization problems defined presents significant challenges to existing mixed-integer solvers. Fortunately, since the problem functions and solutions are not expected to significantly change with the latency under consideration, this is an exercise in the numerical algebra of warm and hot starting. This presents an impetus to develop parametric mixed-integer optimization techniques and software~\cite{bank2021parametric}.



\section{Meso Real-Time Integration}

The Meso-Real-Time Integration layer provides the computational bridge between the 
nonlinear surrogate-augmented meso optimization and the low-latency Real-Time Iteration (RTI) controller. Its primary objective is to convert the nonlinear surrogate dynamics and the refined meso-level particle information into a model-order-reduced (MOR) and locally linearized representation. This representation is then passed to the RTI solver, which computes the actual real-time control action. The Meso-RT integration, therefore, focuses on \emph{model preparation}, not direct real-time optimization. Overall, the proposed framework can be interpreted, in the context of dam networks, as a unified approach to bridging long-term planning, supervisory updates, and real-time hydraulic control.

Throughout this section, $T$ denotes the finite receding prediction horizon
used by the meso and real-time layers, while $T_{\mathrm{final}}$ denotes the
(full) time horizon of the underlying time-dependent PDE-constrained
optimization problem used for conceptual analysis.

\paragraph{Role within the Three-Tier Architecture.}
From the meso layer, we receive the validated surrogate $\hat{\mathcal{S}}_{\theta}$, the refined particle ensemble $\mathcal{P}_{\hat{t}_c}$, and the best-performing local nominal trajectory. The Meso-RT interface uses this information to construct:
\begin{itemize}
    \item a reduced-order model (ROM) via classical or data-driven MOR,
    \item a local linearization of the ROM around the nominal meso trajectory,
    \item an initialization and warm-start guess for RTI.
\end{itemize}
The real-time controller then solves a reduced-dimension quadratic or linear MPC problem 
in a single RTI step. Thus, information flows \emph{from} meso-level adaptation 
\emph{into} real-time execution. The real-time operation will, at the same time, provide streaming data on the physical environment that is used for both meso and offline computational updates.

\paragraph{Operational Data Assimilation and Local Linearization.}
At each real-time instant $\hat{t}_c$, the following procedure updates the model used by RTI:
\begin{enumerate}
    \item Assimilate the measured hydrodynamic state $u^{\hat{\xi}}(\hat{t}_c)$ and update 
          the forecast distribution $\tilde{\rho}(\xi)$.
    \item Identify the best meso particle $(v^{*},z^{*})$ and its corresponding surrogate 
          trajectory.
    \item Construct a reduced-order model 
          \[
             \hat{u}_{\mathrm{ROM}}(t) = A_{\mathrm{ROM}}\,\hat{u}_{\mathrm{ROM}}(t)
             + B_{\mathrm{ROM}}\,v(t) + F_{\mathrm{ROM}} + \text{(small residual)},
          \]
          using MOR techniques (POD, operator inference, autoencoder-based ROM, etc.). Here, the terms $A_{\mathrm{ROM}}$ and $B_{\mathrm{ROM}}$ denote the reduced-order state and control matrices obtained from the MOR procedure, $F_{\mathrm{ROM}}$ denotes the affine forcing/bias term of the reduced model, and $\hat{u}_{\mathrm{ROM}}(t)$ denotes the reduced-order approximation of the hydrodynamic state at time~$t$.
    \item Linearize the ROM around the nominal meso particle:
          \[
             \delta \hat{u}_{t+1}
             = A_t\,\delta \hat{u}_t
             + B_t\,\delta v_t
             + f_t,
          \]
          producing the matrices needed for RTI. Here, the quantity $\delta\hat{u}_t$ denotes the deviation of the reduced state from the nominal meso trajectory, $\delta v_t$ is the deviation of the control input from its nominal value, $A_t$ and $B_t$ are the Jacobians of the ROM with respect to state and control evaluated at the nominal trajectory, and $f_t$ is the affine residual arising from the first-order linearization.
    \item Produce warm-start state and control guesses for RTI (not the final real-time control) using the meso particle and the state deviation 
          from the ROM.
\end{enumerate}

\paragraph{Meso-Real-Time Integration Algorithm.} The integration workflow is summarized below. Note that the algorithm prepares local models and initial guesses for RTI, but does not compute the control action itself; the RTI module handles that step.

\begin{algorithm}[H]
\caption{Meso-Real-Time Integration for Hierarchical Control}
\label{alg:mesort}
\begin{algorithmic}
    \Require Surrogate operator $\hat{\mathcal{S}}_{\theta}$; refined particles $\mathcal{P}_{\hat{t}_c}$;
             policy $\pi_{\phi^{*}}$; tolerance $\epsilon$
    \Ensure Linearized reduced-order model and warm-start for RTI
    \Statex \textbf{(S0$_{\ref{alg:mesort}}$)} Initialize time $t=\hat{t}_c$; load surrogate and meso particles.
    \While{$t < T$}
        \Statex \spc\spc \textbf{(S1$_{\ref{alg:mesort}}$)} Acquire measured state $u^{\hat{\xi}}(t)$ and update uncertainty.
        \Statex \spc\spc \textbf{(S2$_{\ref{alg:mesort}}$)} Select best-performing meso trajectory $(v^*,z^*)$.
        \Statex \spc\spc \textbf{(S3$_{\ref{alg:mesort}}$)} Construct MOR model and obtain ROM matrices.
        \Statex \spc\spc \textbf{(S4$_{\ref{alg:mesort}}$)} Linearize ROM around $(v^*,z^*)$, producing $(A_t,B_t,f_t)$.
        \Statex \spc\spc \textbf{(S5$_{\ref{alg:mesort}}$)} Generate a warm-start $(v_{\mathrm{init}},z_{\mathrm{init}})$ for RTI.
        \Statex \spc\spc \textbf{(S6$_{\ref{alg:mesort}}$)} Compare ROM prediction with measurement: 
        \Statex \spc\spc \spc\spc \spc If discrepancy $>\epsilon$, register data for surrogate retraining.
        \Statex \spc\spc \textbf{(S7$_{\ref{alg:mesort}}$)} Advance $t \leftarrow t+\Delta t$.
    \EndWhile
    \Statex \textbf{(S8$_{\ref{alg:mesort}}$)} Output linearized ROM and warm-start trajectory for RTI.
\end{algorithmic}
\end{algorithm}

The output of Algorithm~\ref{alg:mesort} provides the RTI layer with reduced matrices $(A_t,B_t,f_t)$, a warm-start control trajectory, and updated uncertainty-aware cost estimates. The RTI solver then performs a single SQP step to compute the executable real-time control. Thus, the Meso-RT integration is responsible for model preparation and adaptation, while RTI is responsible for fast optimal control computation.

Observe that the best-performing meso trajectory is expected to be drawn from memory frequently with real-time operation. This permits the consideration of asynchronous methods at the meso layer, i.e., wherein multiple parallel particle evolutions are themselves competing for the best solution or even asynchronous evolution strategies.

\paragraph{Connection to Parareal Real-Time PDE-Constrained Optimization.}
The Meso--Real-Time integration strategy is closely related, at a structural level,
to parareal time-domain decomposition methods for time-dependent
PDE-constrained optimization, as developed by Ulbrich within the real-time
optimization framework~\cite{inbookUlbrich2007}. In that setting, parareal methods are embedded into a generalized SQP algorithm to enable parallel-in-time solution of large-scale transient PDE optimization problems, while retaining the preparation--feedback separation that is central to real-time iteration (RTI) schemes.

In Ulbrich’s formulation, the time horizon $[0,T_{\mathrm{final}}]$ (in contrast to the receding horizon $T$
used by the meso--real-time layers), with
$T_{\mathrm{final}}>0$, is decomposed into subintervals $[T_n,T_{n+1}]$,
where $\{T_n\}_{n=0}^N$ denotes a partition of $[0,T_{\mathrm{final}}]$ satisfying
$T_0=0$, $T_N=T_{\mathrm{final}}$, and $\Delta T := T_{n+1}-T_n$.
The state equation is written in propagator form
\[
  u(T_{n+1}) = g(T_n,u(T_n)),
\]
where $u(t)$ denotes the PDE state, understood as an element of an appropriate
Banach or Hilbert space induced by the spatial discretization of the PDE, and
$g$ denotes the exact time-$\Delta T$ solution operator of the PDE.
A coarse approximation $g^{\Delta}$ of $g$ is assumed available, where
$g^{\Delta}$ denotes a computationally inexpensive coarse-grid approximation
of $g$, typically induced by a lower-order or more dissipative time
discretization.

The parareal update for the state at iteration $k\in\mathbb{N}$ then reads
\begin{equation}
  u^{k+1}(T_{n+1})
  =
  g^{\Delta}\!\left(u^{k+1}(T_n)\right)
  +
  g\!\left(u^{k}(T_n)\right)
  -
  g^{\Delta}\!\left(u^{k}(T_n)\right),
  \label{eq:parareal_ulbrich}
\end{equation}
which can be interpreted as a preconditioned fixed-point iteration for the
multiple-shooting formulation of the time-dependent PDE.
Ulbrich applies this correction simultaneously to the state and adjoint equations
inside a generalized SQP method, allowing inexact and user-provided solvers while
maintaining global convergence.

The Meso--Real-Time interface proposed here follows the same conceptual
separation of responsibilities as the parareal RTI framework, but replaces
adjoint-based correction by value-based and model-based coupling.
Specifically, the meso layer fulfills a role analogous to the coarse global
propagator in parareal--SQP: it provides a globally informed but low-latency
description of future system behavior through the refined particle ensemble
$\mathcal{P}_{\hat t_c}$, surrogate dynamics $\hat{\mathcal{S}}_{\theta}$, and
associated ADP value estimates.

At each control time $\hat t_c$, the Meso--RT layer performs a preparation step in
the precise sense of real-time iteration theory~\cite{Bock2007}.
Here, $t$ denotes the discrete real-time control index associated with the RTI
sampling instants.
The Meso--RT layer constructs a reduced-order model and a local linearization
around the nominal meso trajectory $(v^{*},z^{*})$, yielding
\[
  \delta \hat{u}_{t+1}
  =
  A_t\,\delta \hat{u}_t
  +
  B_t\,\delta v_t
  +
  f_t,
\]
which corresponds to the linearized dynamics required to assemble the quadratic
program solved in the RTI feedback step. The term $f_t$ collects affine residuals arising from the first-order linearization of the reduced-order model around the nominal trajectory. This mirrors the role of the fine propagator in parareal methods, which enforces local consistency with the high-fidelity PDE on short time intervals.

A key difference lies in the coupling mechanism. Whereas parareal--SQP enforces continuity across time subdomains through iterative state and adjoint corrections, the present architecture enforces temporal consistency implicitly through rolling-horizon execution, warm-started RTI, and periodic surrogate validation.
No adjoint sensitivities or second-order information are propagated from the
real-time layer back to the meso layer, reflecting the presence of mixed-integer
decisions, non-smooth dynamics, and real-time latency constraints.

From this perspective, the Meso--Real-Time Integration layer can be interpreted as
a value- and model-based analogue of parareal real-time PDE optimization:
global temporal structure is propagated forward via meso-level particle evolution
and ADP terminal value surrogates, while local PDE fidelity is recovered through
RTI-based linearized MPC at execution time.
This preserves the preparation--feedback philosophy and robustness properties of
classical RTI schemes, while extending parareal ideas to hybrid, stochastic, and
real-time constrained PDE control settings.

\paragraph{Parareal--SQP and 3-Time-Scale ADP.} While Ulbrich’s parareal--SQP framework provides a powerful approach for
time-dependent PDE-constrained optimization, it can not applied directly in the
present setting for several structural reasons.
Parareal--SQP relies on adjoint-based sensitivities, multiple-shooting continuity
corrections, and iterative convergence across the full time horizon, all of which
are incompatible with mixed-integer decisions, non-smooth switching dynamics,
and strict real-time latency constraints. In contrast, the proposed architecture deliberately separates the global structural adaptation from local real-time execution: the meso layer propagates discrete and value information forward in time without adjoints, while the real-time layer executes a single linearized SQP step using reduced-order models. This design preserves the preparation--feedback philosophy of real-time iteration schemes while avoiding the algorithmic overhead and regularity assumptions required by parareal--SQP. At the same time, the use of coarse and fine grid methods, while they need to be adapted to heterogeneous latency, and the MGOpt framework provide a useful set of tools to, for instance, connect the finest linear discretization at the real-time layer through to the coarsest nonlinear discretization at the meso layer, for maximal information efficiency.

\subsection{Open Challenges and Future Research}

The Model Order Reduction from the meso to the real time layer must be quite substantial. Meso layer solutions involve highly nonlinear solutions and must incorporate a long time horizon. At the real-time layer, the problems must be of sufficiently low dimension to be tractable, while at the same time exhibiting reasonable approximation to the real flow. The propagation of approximation and uncertainty accuracy in the real-time from the meso and from the high fidelity offline layer should be investigated theoretically for deriving guarantees of some bounds, while extensive numerical tests will need to be performed validating this accuracy and precision. 


\section{Discussion and Conclusion}

The thought exercise of considering what research and development are required to potentially address the mixed-integer optimization of hydropower cascades with river flow modeled by the shallow water equations, as well as their potential concrete implementation in water engineering, that is, with computational real-time considerations as well as uncertainty, presents a challenging problem in broad applied and computational mathematics. Any realistic implementation requires extensions along the state of the art across many relevant domains - along PDE theory, the calculus of variations, dynamic programming, approximate dynamic programming, model predictive control, uncertainty quantification, reduced order modeling, and scientific computing. Even more significantly, their \emph{integration}, that is, traversing and systematizing across domains and levels of abstraction in a manner that is theoretically sound and computationally tractable, presents uncharted interdisciplinary investigation. As such, this highlights the broad importance, in applied and computational mathematics in general, of the necessity of a systematized understanding of multiple components of the toolkit and their synergy, to continue to leverage the field of applied and computational mathematics for solving real-world problems and continue technological innovation into the future. This problem has, in its nonlinearity and complexity, features exhibited by many contemporary systems of interest, and this paper hopes to be an illustration for and impetus to greater elucidation of methodological interdisciplinary integration towards modeling and control.

\paragraph*{Acknowledgements}
V.K. and M.K. acknowledges funding support from the National Centre for Energy II (TN02000025). KK also acknowledges financial support from the Czech National Science Foundation under Project 24-11664S and M.K. acknowledges the financial support of the Austrian Science Foundation under \url{https://doi.org/10.55776/PAT2747625}. The authors would like to thank Sven Leyffer, Paul Manns, Wladimir Neves and Thomas Surowiec for valuable insight and discussions on the problem.

\end{sloppypar}

\bibliographystyle{plain}
\bibliography{refs}

\end{document}